\documentclass[twoside,11pt]{article}
\usepackage{a4wide}
\usepackage{amsmath,amssymb,amsthm,amsfonts,mathtools,amsbsy}
\usepackage{thmtools}
\usepackage{xcolor}

\usepackage{booktabs}
\usepackage{nicematrix}
\usepackage{xspace}
\newcommand{\ncdof}{\texttt{ncdof}\xspace}
\newcommand{\nnze}{\texttt{nnze}\xspace}
\newcommand{\ndof}{\texttt{ndof}\xspace}
\newcolumntype{L}[1]{>{\raggedright\let\newline\\\arraybackslash\hspace{0pt}}m{#1}}
\newcolumntype{C}[1]{>{\centering\let\newline\\\arraybackslash\hspace{0pt}}m{#1}}
\newcolumntype{R}[1]{>{\raggedleft\let\newline\\\arraybackslash\hspace{0pt}}m{#1}}
\usepackage{array}
\usepackage{comment}

\usepackage{enumitem}
\usepackage{hyperref}
\hypersetup{colorlinks=true,linkcolor=blue!60!black,filecolor=purple!50!black,citecolor=blue!50!green}
\usepackage{url}
\usepackage{accents}
\usepackage[noabbrev,nameinlink,capitalise]{cleveref}
\usepackage{nicefrac}
\usepackage{cite}
\usepackage{xurl}
\usepackage{tikz}
\usetikzlibrary{shapes,arrows,positioning,fit,calc}
\usepackage{fancyhdr}
\usepackage{todonotes}
\newenvironment{eqs} %
 { \begin{equation} \begin{aligned} } %
 { \end{aligned} \end{equation} \ignorespacesafterend } %

\declaretheorem[numberwithin=section]{theorem}
\declaretheorem[sibling=theorem]{lemma}

\declaretheorem[sibling=theorem]{corollary}

\declaretheorem[sibling=theorem]{remark}

\declaretheorem[name=Assumption]{assumption}

\crefname{theorem}{theorem}{theorems}
\Crefname{theorem}{Theorem}{Theorems}
\crefname{lemma}{lemma}{lemmas}
\Crefname{lemma}{Lemma}{Lemmas}
\crefname{proposition}{proposition}{propositions}
\Crefname{proposition}{Proposition}{Propositions}
\crefname{corollary}{corollary}{corollaries}
\Crefname{corollary}{Corollary}{Corollaries}
\crefname{definition}{definition}{definitions}
\Crefname{definition}{Definition}{Definitions}
\crefname{example}{example}{examples}
\Crefname{example}{Example}{Examples}
\crefname{remark}{remark}{remarks}
\Crefname{remark}{Remark}{Remarks}
\crefname{assumption}{assumption}{assumptions}
\Crefname{assumption}{Assumption}{Assumptions}
\let\cref\Cref

\newcommand{\bpm}{\begin{pmatrix}}
\newcommand{\epm}{\end{pmatrix}}

\newcommand{\V}{{\mathbb{V}}}
\newcommand{\W}{{\mathbb{W}}}
\newcommand{\Vh}{{\mathbb{V}_h}}
\newcommand{\Wh}{{\mathbb{W}_h}}
\newcommand{\Vsh}{\hyperlink{def:Vsh}{\V_{\!\star \hspace*{-0.2mm}h}}}
\newcommand{\Wsh}{W_{\!\star \hspace*{-0.2mm}h}}
\newcommand{\Tnorm}[1]{|\!|\!|#1|\!|\!|}
\newcommand{\normDG}[1]{\norm{#1}_{V_h}}
\newcommand{\TnormDG}[1]{{\Vert#1\Vert}_{\Vsh}}
\newcommand*{\abs}[1]{\left|#1\right|}

\newcommand{\jmp}[1]{[\![ #1 ]\!]}

\newcommand{\avg}[1]{\{\!\!\{#1\}\!\!\}}

\newcommand{\Th}{{\calK_h}}
\newcommand{\Fh}{{\calF_h}}
\newcommand{\Fhi}{{\calF_h^i}}

\newcommand{\nf}{{n_F}}

\newcommand{\alphamin}{{\alpha_{\min}}}

\newcommand{\VKh}{\mathbb{V}_h(K)}
\newcommand{\Bh}{\hyperlink{def:Bh}{B_h}}
\newcommand{\Wth}{\hyperlink{def:Wth}{\tilde{W}_h}}
\newcommand{\Uh}{\hyperlink{def:Uh}{\mathbb{L}_h}}
\newcommand{\Uhnl}{\mathbb{L}_h}

\newcommand{\UKh}{\hyperlink{def:Uh}{\mathbb{L}_h(K)}}

\newcommand{\QT}{\hyperlink{def:QT}{\mathbb{Q}(\Th)}}
\newcommand{\QTp}{\hyperlink{def:QTp}{\mathbb{Q}(\Th)'}}
\newcommand{\QTh}{\hyperlink{def:QTh}{\mathbb{Q}_h(\Th)}}
\newcommand{\QThp}{\hyperlink{def:QThp}{\mathbb{Q}_h(\Th)'}}
\newcommand{\AK}{A_K}
\newcommand{\ATh}{\hyperlink{def:ATh}{A_{\Th}}}
\newcommand{\AThnorm}[2]{\hyperref[eq:AThnorm]{\vert} #1 \hyperref[eq:AThnorm]{\vert_{A_{\Th}}^{#2}}}
\newcommand{\ahnorm}[2]{\hyperref[eq:ahnorm]{\vert} #1 \hyperref[eq:ahnorm]{\vert_{a_h}^{#2}}}
\newcommand{\Bhnorm}[2]{\hyperref[eq:Bhnorm]{\Vert} #1 \hyperref[eq:Bhnorm]{\Vert_{B_h}^{#2}}}
\newcommand{\QK}{\mathbb{Q}(K)}
\newcommand{\QKh}{\mathbb{Q}_h(K)}

\DeclareMathOperator{\tc}{T}
\DeclareMathOperator{\tch}{\hyperlink{def:tch}{T_h}}
\DeclareMathOperator{\Div}{div}

\DeclareMathOperator{\Lap}{\Delta}

\DeclareMathOperator{\id}{id}

\DeclareMathOperator{\diam}{diam}

\DeclareMathOperator{\diag}{diag}
\renewcommand{\dim}{\operatorname{dim}}

\newcommand{\inner}[1]{( #1 )}

\newcommand*{\norm}[1]{\left\|#1\right\|}

\newcommand\restr[2]{{ \left.\kern-\nulldelimiterspace #1 \vphantom{\big|} \right|_{#2} }}

\newcommand{\IH}{\mathbb{H}}

\newcommand{\IL}{\mathbb{L}}

\newcommand{\IN}{\mathbb{N}}

\newcommand{\IP}{\mathbb{P}}

\newcommand{\IR}{\mathbb{R}}

\newcommand{\IT}{\mathbb{T}}
\newcommand{\ITh}{\hyperlink{def:ITh}{\mathbb{T}_h}}

\newcommand{\calF}{\mathcal{F}}

\newcommand{\calH}{\mathcal{H}}

\newcommand{\calK}{\mathcal{K}}

\newcommand{\calT}{\mathcal{T}}

\newcommand{\bi}{\mathbf{i}}

\newcommand{\bell}{{\boldsymbol{\ell}}}

\newcommand{\bA}{\mathbf{A}}

\newlength{\squeezedstackrelawd}
\newlength{\squeezedstackrelbwd}
\def\squeezedstackrel#1#2{\settowidth{\squeezedstackrelawd}%
{${{}^{#1}}$}\settowidth{\squeezedstackrelbwd}{$#2$}%
\addtolength{\squeezedstackrelawd}{-\squeezedstackrelbwd}%
\leavevmode\ifthenelse{\lengthtest{\squeezedstackrelawd>0pt}}%
{\kern-.5\squeezedstackrelawd}{}\mathrel{\mathop{\vphantom{|}#2}\limits^{\text{\footnotesize$\underbrace{#1}$}}}\ifthenelse{\lengthtest{\squeezedstackrelawd>0pt}}%
{\kern-.5\squeezedstackrelawd}{}}

\definecolor{pscol}{rgb}{0.8,0,0}

\usepackage{pgfplots} 
\pgfplotsset{compat=newest}
\usepgfplotslibrary{colorbrewer,groupplots}
\pgfplotsset{
	discard if not/.style 2 args={
	x filter/.append code={
	\edef\tempa{\thisrow{#1}}
	\edef\tempb{#2}
	\ifx\tempa\tempb
	\else
	
	\fi
	}
	}
}
\usepackage{pgfplotstable} 
\usepackage{booktabs} 
\usepackage{colortbl}
\makeatletter
\pgfplotstableset{
	discard if not/.style 2 args={
	row predicate/.append code={
	\def\pgfplotstable@loc@TMPd{\pgfplotstablegetelem{##1}{#1}\of}
	\expandafter\pgfplotstable@loc@TMPd\pgfplotstablename
	\edef\tempa{\pgfplotsretval}
	\edef\tempb{#2}
	\ifx\tempa\tempb
	\else
	\pgfplotstableuserowfalse
	\fi
	}
	}
}
\makeatother

\pgfplotscreateplotcyclelist{paulcolors2}{%
	orange,every mark/.append style={solid,fill=orange},mark=diamond*,very thick,mark size=3pt\\%
	orange,densely dashed,every mark/.append style={solid,fill=orange},mark=diamond*,very thick,mark size=3pt\\%
	teal,every mark/.append style={solid,fill=teal},mark=*,very thick,mark size=3pt\\%
	teal,densely dashed,every mark/.append style={solid,fill=teal},mark=*,very thick,mark size=3pt\\%
	violet,every mark/.append style={solid,fill=violet},mark=square*,very thick,mark size=3pt\\%
	violet,densely dashed,every mark/.append style={solid,fill=violet},mark=square*,very thick,mark size=3pt\\%
	cyan,every mark/.append style={solid,fill=cyan},mark=star,very thick,mark size=3pt\\%
	cyan,densely dashed,every mark/.append style={solid,fill=cyan},mark=star,very thick,mark size=3pt\\%
}

\pgfplotscreateplotcyclelist{paulcolors3}{%
	orange,every mark/.append style={solid,fill=orange},mark=diamond*,very thick,mark size=3pt\\%
	orange,densely dashed,every mark/.append style={solid,fill=orange},mark=diamond*,very thick,mark size=3pt\\%
	orange,dotted,every mark/.append style={solid,fill=orange},mark=diamond*,very thick,mark size=3pt\\%
	teal,every mark/.append style={solid,fill=teal},mark=*,very thick,mark size=3pt\\%
	teal,densely dashed,every mark/.append style={solid,fill=teal},mark=*,very thick,mark size=3pt\\%
	teal,dotted,every mark/.append style={solid,fill=teal},mark=*,very thick,mark size=3pt\\%
	violet,every mark/.append style={solid,fill=violet},mark=square*,very thick,mark size=3pt\\%
	violet,densely dashed,every mark/.append style={solid,fill=violet},mark=square*,very thick,mark size=3pt\\%
	violet,dotted,every mark/.append style={solid,fill=violet},mark=square*,very thick,mark size=3pt\\%
	cyan,every mark/.append style={solid,fill=cyan},mark=star,very thick,mark size=3pt\\%
	cyan,densely dashed,every mark/.append style={solid,fill=cyan},mark=star,very thick,mark size=3pt\\%
	cyan,dotted,every mark/.append style={solid,fill=cyan},mark=star,very thick,mark size=3pt\\%
}

\pgfplotscreateplotcyclelist{paulcolors321}{%
	orange,every mark/.append style={solid,fill=orange},mark=diamond*,very thick,mark size=3pt\\%
	orange,densely dashed,every mark/.append style={solid,fill=orange},mark=diamond*,very thick,mark size=3pt\\%
	orange,dotted,every mark/.append style={solid,fill=orange},mark=diamond*,very thick,mark size=3pt\\%
	teal,every mark/.append style={solid,fill=teal},mark=*,very thick,mark size=3pt\\%
	teal,densely dashed,every mark/.append style={solid,fill=teal},mark=*,very thick,mark size=3pt\\%
	violet,every mark/.append style={solid,fill=violet},mark=square*,very thick,mark size=3pt\\%
}

\title{
Embedded Trefftz DG framework for the analysis of discretizations with local-global decompositions
}
\author{ 
Philip L. Lederer\thanks{Faculty of Mathematics, Informatics and Natural Sciences, University of Hamburg, Germany} 
\and 
Christoph Lehrenfeld\thanks{Institut für Numerische und Angewandte Mathematik, Georg-August Universität Göttingen, Germany}
\and 
Paul Stocker\thanks{Faculty of Mathematics, University of Vienna, Austria} 
\and 
Igor Voulis\footnotemark[2]
}
\date{}

\begin{document}
\maketitle

\begin{abstract}
This paper presents a framework for the analysis of discretization methods based on the decomposition into local and global problems. 
We apply the framework to provide a comprehensive error analysis for the embedded Trefftz discontinuous Galerkin method, for a wide range of second-order scalar elliptic partial differential equations and a scalar reaction-advection problem.
We also analyze quasi-Trefftz methods with our framework, presenting the first optimal error bounds in weaker norms.
\end{abstract}

\noindent\textbf{Keywords:} {Trefftz methods, discontinuous Galerkin methods, error analysis, embedded Trefftz, quasi-Trefftz}\\
\noindent\textbf{MSC:} 
65N12, 
65N30, 
65M12, 
65M60, 
65J05 

\section{Introduction}
Trefftz methods, named after Erich Trefftz \cite{trefftz1926}, are a class of numerical methods that use solutions of the governing linear partial differential equation (PDE) as basis functions for the discretization.
While various approaches exist for implementing this ansatz, it is particularly well-suited for application within the framework of discontinuous Galerkin (DG) methods. In this context, the Trefftz space is utilized as the local test and trial spaces, see e.g. \cite{HMPS14,mope18,bgl2016,bcds20,SpaceTimeTDG,KSTW2014,PSSW_CMA_2020,HLSW_M2AN_2022,LLS_NM_2024,GMPS_AML_2023,gomo22}.
Consequently, it offers a promising alternative to other approaches for reducing unknowns, especially in the context of DG methods.
On polytopal meshes strong advantages over other DG approaches
such as standard DG, Hybrid-DG and Hybrid High Order methods as well as the Virtual Element Method are observed, see \cite{LSZ_PAMM_2024}.
Unfortunately, depending on the governing PDE, the resulting (local) Trefftz spaces are not always polynomial.
In most cases in the literature only PDE problems with constant coefficients and homogeneous right-hand sides are considered as only then  
a suitable Trefftz space can be constructed. 

Efforts have been made to extend the Trefftz paradigm to more general PDEs, including non-constant coefficients and non-zero right-hand sides by relaxing the condition on the discrete spaces: the quasi-Trefftz method and the embedded Trefftz method.
Both approaches allow to extend the class of PDE problems that can efficiently be treated by Trefftz DG methods significantly.

The quasi-Trefftz method, see \cite{imbert2014generalized,IGMS_MC_2021,2408.00392,gomo24}, relies on a relaxation of the Trefftz condition, i.e. that basis functions are solutions of the governing PDE, by demanding a condition based on the Taylor expansion of the PDE coefficients and the right-hand side on selected points.

The recently proposed embedded Trefftz method \cite{LS_IJMNE_2023, lozinski19} introduces a more general relaxation of the Trefftz condition based on projections, resulting in \emph{weak Trefftz spaces}, and only requires easy-to-compute local problems to be solved, while completely avoiding the explicit construction of Trefftz functions.
The embedded Trefftz method has been successfully applied to a wide range of PDEs, including Poisson equation, acoustic wave equation, Helmholtz equation, and a linear transport equation \cite{LS_IJMNE_2023}, as well as Stokes equation \cite{LLS_NM_2024}.
However, the analysis of the embedded Trefftz method is still in its infancy, and only a few error estimates are available. 
See \cite{lozinski19} for the Poisson problem with varying coefficients, and \cite{LLS_NM_2024} for the Stokes problem with inhomogeneous right-hand sides but constant coefficients.

The earlier work \cite{lozinski19} by A. Lozinski already provides a method similar to \cite{LS_IJMNE_2023} under the name of `discontinuous Galerkin method with static condensation' for the Poisson problem with varying coefficients. 
    It includes an a-priori error analysis for this case.
    Up to now, this work has sadly been overlooked by the Trefftz community.
    One key difference between the methods is the construction of the local spaces, which in \cite{lozinski19} are constructed by solving local problems and orthonormalizing the basis vectors using Gram-Schmidt.

Across all three approaches, the full problem splits into element-wise local problems and a global problem formulated on the respective global -- respectively, generalized \emph{Trefftz} -- space.

\paragraph{Contributions.}
It turns out that the overall setting of the different approaches can be unified in a general framework.
The main contribution of this work is as follows:

\begin{itemize}

    \item We develop a general analysis framework for discretization methods with a local-global decomposition and provide precise assumptions for the a-priori error analysis, along with tools to verify them.
    Crucially, the framework derives its error bounds from the approximation properties of the full discrete function space -- comprising both the collection of local spaces and the global (or \emph{Trefftz}) space -- rather than from the Trefftz space alone.

    \item The framework enables a unified analysis of embedded Trefftz DG methods, revealing their structure and supporting extensions to more complex settings. 
        It allows us to extend the analysis in \cite{lozinski19} to a wider class of PDEs.
        Based on the construction of a stable space decomposition, we propose a generalization of the embedded Trefftz DG method.

    \item Applying the framework to the quasi-Trefftz DG approach we provide new error bounds in weaker norms.
        Compared to existing analysis, we avoid restrictive interpolation arguments used in standard Trefftz analyses, as the error bounds in the framework lead to optimal error bounds in standard polynomial spaces.
\end{itemize}

The framework applies broadly to Trefftz DG and other methods that can be decomposed into local and global problems.
It lays the groundwork for future (embedded) Trefftz DG generalizations and applications to more complex PDEs.
Its application to Trefftz collocation methods and other spectral schemes \cite{llhc08,kk95} is limited and not the focus of this work.

We emphasize that the strength of the framework lies in the derivation of error bounds by the best approximation of an \emph{underlying discrete space}, such as the full polynomial space.
This is paramount whenever \emph{best approximation estimates for the Trefftz space are not available}, e.g. for the embedded Trefftz DG method.

\paragraph{Outline.}
We start by motivating the framework, introducing several methods that fit into it, in \cref{sec:Tlike}.
Specific choices and possible constructions of a splitting of the discretization space into local and global parts are discussed.
Especially, several variants of \emph{Trefftz} DG methods are discussed.

The theoretical backbone of the paper is presented in \cref{sec:framework} where we introduce the general framework for discretization methods that can be decomposed into two parts in the following manner: 
One part consists of a set of \emph{local} subproblems and corresponding local unknowns, both associated to elements in an index set, typically the elements in a computational mesh. 
The remainder consists of a \emph{global} problem with globally coupled unknowns.

In our general approach we assume that both problems are coupled with each other and form a $2\times 2$ block system.
However, in many cases, such as the embedded Trefftz DG method, the problems can be decoupled, i.e. the local subproblems can be solved independently of the global problem.

The key result in \cref{sec:framework} is the stability analysis for this coupled system (\cref{thm:Tcoercivity:coupled}) which is based on three essential assumptions of the framework: 
\begin{itemize}
\item Stability of the local subproblems (\cref{ass:local}).
\item Stability of the global problem (\cref{ass:assumption_Th}).
\item An assumption on a sufficiently weak coupling between the local and global problems -- at least in one direction (\cref{ass:inexacttrefftz}).
\end{itemize}
Depending on the discretization setting these three assumptions can be difficult to verify, especially the stability of the local subproblems (\cref{ass:local}). 

In \cref{sec:tools} we hence present a set of more accessible sufficient conditions that allow to deduce the stability of the local subproblems. 
The main components of the analysis covered in \cref{sec:framework,sec:tools} and their interdependence is outlined in \cref{fig:overview}.
\begin{figure}[!tb]
    \begin{center}
        \scalebox{0.69}{\pgfdeclarelayer{background}%
\pgfsetlayers{background,main}%
\begin{tikzpicture}[
    line/.style={draw, -latex'},
    block/.style={rectangle, draw, fill=gray!20, text width=8em, text centered, minimum height=6em}
]

\node[block, fill=green!20] (A2) {
stability of local problems \\[1ex]    
(\cref{ass:local})
};

\node[block, fill=blue!20, below=of A2] (A4) {
stability of global problems \\[1ex]
(\cref{ass:assumption_Th}) 
};

\node[block, fill=orange!20, below=of A4] (A6) {
sufficiently weak coupling \\[1ex]
(\cref{ass:inexacttrefftz})
};

\node[block, fill=red!20, right=1cm of A2] (Bhstab) { 
    stability of coupled problem \\[1ex] 
    (\cref{thm:Tcoercivity:coupled})
};

\node[block, fill=blue!20, below=of Bhstab] (A5) {
continuity of global operator \\[1ex]    
(\cref{ass:cont_Th})
};

\node[block, fill=green!20, below=of A5] (A3) {
    continuity of local operator \\[1ex]   
    (\cref{ass:local_strong})
};

\node[block, fill=red!20, right=1cm of $(Bhstab.south east)!0.5!(A5.north east)$] (Cea) { 
    Strang-type result \\[1ex]
    (\cref{cor:cea})
};


\node[block, fill=green!20, left=1cm of A2] (eq20a) {
local decomposition of local part \\[1ex]
\eqref{eq:quasiortho:a}
};

\node[block, fill=green!20, below=of eq20a] (eq20b) {
continuous space decomposition of local part \\[1ex]
    \eqref{eq:quasiortho:b} 
};

\node[block, fill=green!20, below=of eq20b] (eq20c) {
local stability of local problems \\[1ex]
\eqref{eq:quasiortho:c} 
};

\node[block, fill=orange!20, left=1cm of eq20a] (A8) {
locality of local operator \\[1ex]
(\cref{ass:locality})
};

\node[block, fill=orange!20, below=of A8] (DG) {
DG setting \\[1ex]
(\cref{ssec:DGsetting}) 
};

\node[block, fill=green!20, below=of DG] (AK0) {
    Existence of suitable prototype operator \\[1ex]
    (\cref{lem:abstractneumann})
};





\begin{pgfonlayer}{background}
\node[fill=gray!0, scale=1, below=0.02cm of $(eq20c.south west)!0.5!(AK0.south east)$] (tools) { 
    \cref{sec:tools}
};
\node[fill=gray!0, scale=1, below=0.02cm of $(A3.south)$] (framework) { 
    \cref{sec:framework}
};
\node[draw=gray, dashed, fill=violet, opacity=0.1, draw opacity=1.0,fit=(A2) (A3) (A4) (A5) (A6) (Bhstab) (Cea),inner sep=0.5cm] (G1) {};
\node[draw=gray, dashed, fill=teal, opacity=0.15, draw opacity=1.0, fit=(DG) (AK0) (eq20a) (eq20b) (eq20c),inner sep=0.5cm] (eq20) {};
\end{pgfonlayer}


\node[block, fill=green!20, minimum height=2em, below=.25cm of $(eq20.290)$] (legend1) { local problems };
\node[block, fill=blue!20, minimum height=2em, right=.5cm of $(legend1.east)$] (legend2) { global problems };
\node[block, fill=orange!20, minimum height=2em, right=.5cm of $(legend2.east)$] (legend3) { coupling};
\node[block, fill=red!20, minimum height=2em, right=.5cm of $(legend3.east)$] (legend4) { result };
\node[block, fill=none, text width=19.5em, draw=none, minimum height=2em, left=.5cm of $(legend1.east)$] (legend0) { \textbf{legend:} };

\path[line] (A2.east) to[in=180,out=0] (Bhstab.170);
\path[line] (A4.20) to[in=210,out=60] (Bhstab.190);
\path[line] (A6.east) to[in=230,out=65] (Bhstab.210);

\path[line] (Bhstab.east) to[in=195,out=-10] (Cea.170);
\path[line] (A5.east) to[in=210,out=45] (Cea.190);
\path[line] (A3.east) to[in=230,out=60] (Cea.210);

\path[line] (eq20a.east) to[in=180,out=0]  (A2.170);
\path[line] (eq20b.20) to[in=210,out=60] (A2.190);
\path[line] (eq20c.east) to[in=230,out=65] (A2.210);

\path[line] (A8.0) to[] (eq20a.180);
\path[line] (DG.0) to[] (eq20b.180);

\path[line] (AK0.0) to[in=180,out=0] (eq20c.180);

\end{tikzpicture}}
    \end{center} 
    \vspace*{-0.6cm}
    \caption{Overview of the main components in the analysis of the considered framework.} \label{fig:overview}
\end{figure}
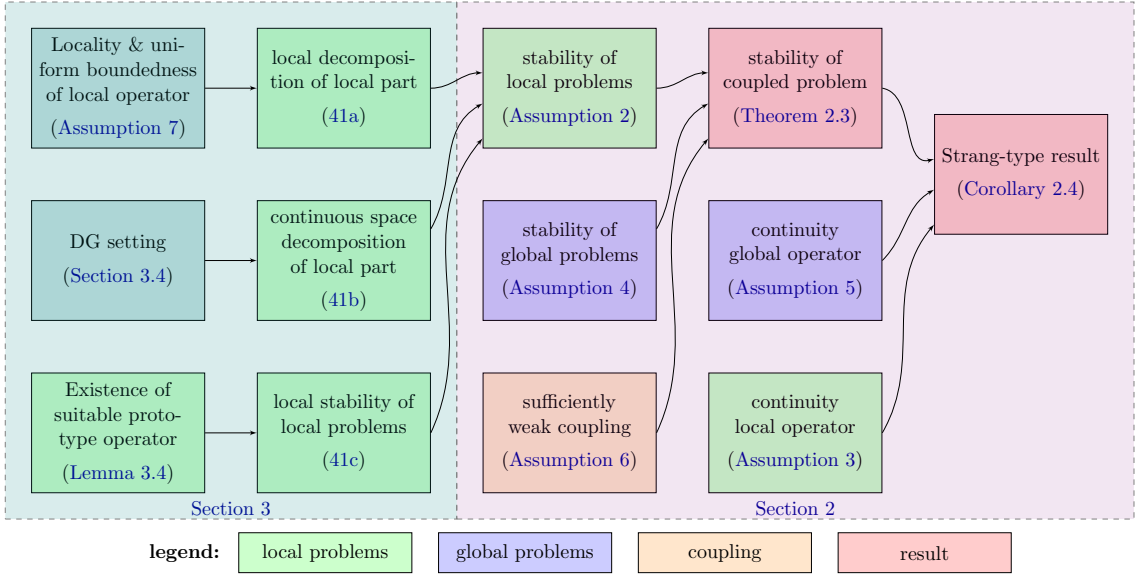

While the discussion of methods within the framework in \cref{sec:Tlike} is on a rather conceptual level, in \cref{sec::examples} we consider several applications of Trefftz DG methods.
We verify the assumptions of the framework and hence derive optimal a-priori error bounds for concrete PDE problems.

\section{Examples of methods within the framework} \label{sec:Tlike}

In this section we discuss methods that fit naturally within the presented framework.
The popularity of such methods stems from the reduction of unknowns by considering suitable subspaces of well known CG and DG spaces.
By applying the results of this work, these methods inherit the optimal approximation properties of the underlying spaces.

To motivate the methods, we consider a mesh $\Th$ consisting of non-overlapping elements $K$, more precise assumptions on the mesh are given in \cref{sec::examples}.
In this section we consider $\Vh$ to be the piecewise polynomial space of degree $p$, $\Vh = \{v\in L^2(\Omega) \mid v|_K \in \IP^p(K),\ \forall K\in\Th\}$.

\subsection{Motivation} 

We want to briefly motivate our interest in Trefftz methods by comparing the number of degrees of freedom (\ndof) and the number of non-zero matrix entries (\nnze) in the linear system for several DG-based methods in \cref{tab:motivation}, even though we are of course aware that these measures reflect only one of several aspects relevant to the performance of a discretization method. An in-depth comparison in terms of degrees of freedom and number of non-zero entries between Trefftz DG methods and other polytopal methods such as standard DG, Virtual Element Method, and Hybrid-DG can be found in \cite{LSZ_PAMM_2024}. For methods involving static condensation\footnote{cf. \cref{ssec:staticcond} for a brief explanation of static condensation} the numbers are computed for the condensed system where we denote the remaining number of coupling degrees of freedom as \ncdof.

The methods considered are standard discontinuous Galerkin (DG), Trefftz DG (TDG), hybrid DG (HDG) methods.
We consider two Trefftz DG versions, one for a first order PDE operator (TDG1), e.g. for an advection-reaction equation and one for a second order PDE operator (TDG2), e.g. for the diffusion equation.

For the DG method we further distinguish two realizations w.r.t. the choice of degrees of freedom or basis functions, respectively.  
In the realization denoted as standard DG (DG), we assume the worst case in terms of couplings, i.e. every degree of freedom is coupled with all element degrees of freedom of the same and all neighboring elements.
Also for the Trefftz DG methods we assume this worst case.
In the \emph{compact} DG method, cf.  
\cite{zbMATH05582044}, the number of basis functions with support on facets is minimized which minimizes the couplings across element interfaces alongside. 
Element bubbles are then condensated.

\begin{table}[!ht]
\centering
\begin{NiceTabular}{L{2.65cm}@{}R{1.0588235294117647cm}@{}R{1.1176470588235294cm}@{}R{1.1764705882352942cm}@{}R{1.2352941176470589cm}@{}R{1.2941176470588236cm}@{}R{1.3529411764705883cm}@{}R{1.4117647058823528cm}@{}R{1.4705882352941178cm}@{}R{1.5294117647058822cm}@{}}[colortbl-like]
    \toprule
\text{method}$~\!\downarrow$ $\!\setminus\!~k\!\! \rightarrow$ & \multicolumn{1}{r}{1}& \multicolumn{1}{r}{2}& \multicolumn{1}{r}{3}& \multicolumn{1}{r}{4}& \multicolumn{1}{r}{5}& \multicolumn{1}{r}{6}& \multicolumn{1}{r}{7}& \multicolumn{1}{r}{8}& \multicolumn{1}{r}{9}\\
   \midrule
$\texttt{ndof}_{\text{DG}}/N_{\text{El}}$&         4\hphantom{.0} &         10\hphantom{.0} &         20\hphantom{.0} &         35\hphantom{.0} &         56\hphantom{.0} &         84\hphantom{.0} &        120\hphantom{.0} &        165\hphantom{.0} &        220\hphantom{.0}\\
\rowcolor{gray!20}$\texttt{ncdof}_{\text{CDG}}/N_{\text{El}}$&         4\hphantom{.0} &         10\hphantom{.0} &         20\hphantom{.0} &         34\hphantom{.0} &         52\hphantom{.0} &         74\hphantom{.0} &        100\hphantom{.0} &        130\hphantom{.0} &        164\hphantom{.0}\\
$\texttt{ncdof}_{\text{HDG}}/N_{\text{El}}$&         6\hphantom{.0} &         12\hphantom{.0} &         20\hphantom{.0} &         30\hphantom{.0} &         42\hphantom{.0} &         56\hphantom{.0} &         72\hphantom{.0} &         90\hphantom{.0} &        110\hphantom{.0}\\
\rowcolor{gray!20}$\texttt{ndof}_{\text{TDG2}}/N_{\text{El}}$&         4\hphantom{.0} &          9\hphantom{.0} &         16\hphantom{.0} &         25\hphantom{.0} &         36\hphantom{.0} &         49\hphantom{.0} &         64\hphantom{.0} &         81\hphantom{.0} &        100\hphantom{.0}\\
$\texttt{ndof}_{\text{TDG1}}/N_{\text{El}}$&         3\hphantom{.0} &          6\hphantom{.0} &         10\hphantom{.0} &         15\hphantom{.0} &         21\hphantom{.0} &         28\hphantom{.0} &         36\hphantom{.0} &         45\hphantom{.0} &         55\hphantom{.0}\\
\rowcolor{gray!20}$\texttt{ncdof}_{\text{FEM}}/N_{\text{El}}$&       0.2 &        3.3 &        8.5 &       15.7 &       24.8 &         36\hphantom{.0} &       49.2 &       64.3 &       81.5\\
    \bottomrule
\end{NiceTabular}

\begin{NiceTabular}{L{2.65cm}@{}R{1.0588235294117647cm}@{}R{1.1176470588235294cm}@{}R{1.1764705882352942cm}@{}R{1.2352941176470589cm}@{}R{1.2941176470588236cm}@{}R{1.3529411764705883cm}@{}R{1.4117647058823528cm}@{}R{1.4705882352941178cm}@{}R{1.5294117647058822cm}@{}}[colortbl-like]
    \toprule
\text{method}$~\!\downarrow$ $\!\setminus\!~k\!\! \rightarrow$ & \multicolumn{1}{r}{1}& \multicolumn{1}{r}{2}& \multicolumn{1}{r}{3}& \multicolumn{1}{r}{4}& \multicolumn{1}{r}{5}& \multicolumn{1}{r}{6}& \multicolumn{1}{r}{7}& \multicolumn{1}{r}{8}& \multicolumn{1}{r}{9}\\
   \midrule
$\texttt{nnze}_{\text{DG}}/N_{\text{El}}$&        80\hphantom{.0} &        500\hphantom{.0} &       2000\hphantom{.0} &       6125\hphantom{.0} &      15680\hphantom{.0} &      35280\hphantom{.0} &      72000\hphantom{.0} &     136125\hphantom{.0} &     242000\hphantom{.0}\\
\rowcolor{gray!20}$\texttt{nnze}_{\text{CDG}}/N_{\text{El}}$&        64\hphantom{.0} &        340\hphantom{.0} &       1200\hphantom{.0} &       3196\hphantom{.0} &       7072\hphantom{.0} &      13764\hphantom{.0} &      24400\hphantom{.0} &      40300\hphantom{.0} &      62976\hphantom{.0}\\
$\texttt{nnze}_{\text{HDG}}/N_{\text{El}}$&       126\hphantom{.0} &        504\hphantom{.0} &       1400\hphantom{.0} &       3150\hphantom{.0} &       6174\hphantom{.0} &      10976\hphantom{.0} &      18144\hphantom{.0} &      28350\hphantom{.0} &      42350\hphantom{.0}\\
\rowcolor{gray!20}$\texttt{nnze}_{\text{TDG2}}/N_{\text{El}}$&        80\hphantom{.0} &        405\hphantom{.0} &       1280\hphantom{.0} &       3125\hphantom{.0} &       6480\hphantom{.0} &      12005\hphantom{.0} &      20480\hphantom{.0} &      32805\hphantom{.0} &      50000\hphantom{.0}\\
$\texttt{nnze}_{\text{TDG1}}/N_{\text{El}}$&        45\hphantom{.0} &        180\hphantom{.0} &        500\hphantom{.0} &       1125\hphantom{.0} &       2205\hphantom{.0} &       3920\hphantom{.0} &       6480\hphantom{.0} &      10125\hphantom{.0} &      15125\hphantom{.0}\\
\rowcolor{gray!20}$\texttt{nnze}_{\text{FEM}}/N_{\text{El}}$&       2.5 &      103.5 &        488\hphantom{.0} &       1408\hphantom{.0} &     3187.5 &     6222.5 &      10981\hphantom{.0} &      18003\hphantom{.0} &    27900.5\\
    \bottomrule
\end{NiceTabular}

\caption{\ncdof and \nnze per element for different methods on periodic tetrahedral mesh for
$k\!\!=\!\!1,\!..,\!10$ (rounded up to one decimal place).}
\label{tab:motivation}
\end{table}

In \cref{tab:motivation} we present a comparison of the \ndof on 
a tetrahedral mesh for problems with periodic boundary condition.
While all methods considered reduce the order of the dimension of the standard DG method from $\mathcal{O}(p^d)$ to $\mathcal{O}(p^{d-1})$ the constant varies significantly.
Clearly HDG and Trefftz DG methods are the most efficient methods in terms of \ndof behind FEM.

\subsection{Classical polynomial Trefftz DG}

Traditionally, Trefftz DG methods are considered for homogeneous PDEs, where the right-hand side is zero. 
Polynomial Trefftz DG methods reduce the standard piecewise polynomial space $\Vh$ by considering only basis functions that are locally in the kernel of the differential operator of the PDE.
Then, the Trefftz space on a mesh element $K$ is given by
\begin{equation}\label{eq:ITh}
    \IT_h = \{ u \in \Vh \text{ s.t. } \AK  \restr{u}{K} = 0\quad \forall K\in\Th\},
\end{equation}
where $A_K$ denotes the (local) differential operator.
As an example, consider the Laplace equation $-\Delta u = 0$, then $A_K = -\Delta|_K$. Hence $\IT_h$ is given by the harmonic polynomials. 

Our framework allows to analyze the approximation properties of these Trefftz spaces using the best approximation in $\Vh$.

\begin{remark}[Plane wave DG]
    A popular application of Trefftz methods are time-harmonic wave problems, 
    see the survey \cite{TrefftzSurvey}.
    There, the Trefftz space $\IT_h$ is given by non-polynomial wave functions. 
    This space is not know to be a subspace of a larger space $\Vh$ with well-studied approximation properties.    
    Thus, our framework does not immediately provide any additional insight.
\end{remark}

\subsection{Embedded Trefftz DG}\label{sec:embtrefftz}

The embedded Trefftz method, presented in \cite{LS_IJMNE_2023}, solves a Trefftz DG problem by constructing an embedding into a standard polynomial DG method. Earlier, in \cite{lozinski19}, a very similar method was proposed under the name of discontinuous Galerkin method with static condensation.
The method can be applied to challenging PDE operators of varying order and non-constant coefficients by constructing the embedding for a Trefftz-like space with a weaker Trefftz property.
In \cite{LS_IJMNE_2023} promising numerical results for Laplace equation, Poisson equation, acoustic wave equation with piecewise constant and also with smooth coefficient, Helmholtz equation, and a linear transport equation are presented.
In \cite{lozinski19} the method was applied to the Poisson problem with varying coefficients and an a-priori error analysis was provided for this case.

Similar to \eqref{eq:ITh}, the method solves the PDE over (weak) Trefftz spaces $\IT_h$ given by 
\begin{equation}\label{eq:IThproto}
    \IT_h = \{ u \in \Vh \text{ s.t. } \langle \AK  u, q \rangle_K = 0\quad \forall q \in \QKh,\ K\in\Th\},
\end{equation}
for a suitably chosen space $\QKh$, which is typically a polynomial space of lower order than the polynomial space $\Vh(K)$.
The operator $\AK $ is the local differential operator of the PDE.
The pairing $\langle \cdot,\cdot\rangle_K$ is a suitable dual pairing on the element $K$, usually the $L^2$-inner product.

If the choice $\QKh = \AK \Vh(K)$ leads to suitable Trefftz space, then this method is equivalent to the classical Trefftz DG method. By relaxing the condition with $\QKh$, the embedded Trefftz method generalizes the classical Trefftz DG method to a wider class of PDEs, including non-constant coefficients. 

Furthermore, the embedded Trefftz method allows a non-vanishing right-hand side. In this case, a suitable affine shift of the Trefftz space is constructed by solving local problems on each mesh element $K$. This is done by taking a local complement $\IL_h(K)$ of the Trefftz space $\IT_h$ in the polynomial space $\Vh$ for each element $K\in\Th$, and solving the local problem
\[
    \langle \AK  u, q \rangle_K = \langle \ell_K, q \rangle_K\quad \forall q \in \QKh,\ K\in\Th,
\]
where $\ell_K$ is the restriction of the right-hand side to the element $K$.

\subsection{Quasi-Trefftz DG}\label{sec:qtrefftz}
Quasi-Trefftz DG methods, see e.g. \cite{yang2020trefftz,IGMS_MC_2021,2408.00392}, are another generalization of the Trefftz DG method to PDEs with varying coefficients and right-hand sides. 
Basis functions for a quasi-Trefftz spaces $\IT_h$ and solutions of the local problems are constructed methodically to satisfy
\begin{equation}\label{eq:qtspace}
    \IT_h =\big\{ u\in \Vh \text{ s.t. } D^{\bi} A_K u (x_K) = 0\quad  \forall \bi\in \IN^d_0,\ |\bi|\leq p-m,\ K\in\Th \big\},
\end{equation}
where $A_K$ is the local differential operator of the PDE and $m$ is the order of the PDE.
Here $D^{\bi}$ is the partial derivative with respect to the multi-index\footnote{  Multi-indices are denoted $\bi:=(i_{1},\ldots,i_{d}) \in \IN_0^{d}$ and their length $|\bi|:=i_{1}+\cdots+i_{d}$.} $\bi$.
Here the Trefftz-condition is weakened by requiring that a truncated Taylor expansion of the PDE operator $A_K$ vanishes at a given point $x_K$ in the element $K$.

Like the embedded Trefftz method, the quasi-Trefftz method allows to treat non-homogeneous PDEs by solving local problems on each element $K$. In this case, the local problem reads
\[
D^{\bi} A_K u_{\IL,K}(x_K)  = D^\bi f(x_K) \quad \forall \bi\in \IN^d_0,\ |\bi|\leq p-m,\ K\in \Th, 
\]
which is solved by a recursive procedure which relies on the Taylor expansion of the PDE operator $A_K$ and the right-hand side $f$.

\subsection{A new perspective on static condensation} \label{ssec:staticcond}
The static condensation method is a technique to reduce the size of the linear system of equations by eliminating the degrees of freedom associated with the interior of the elements and is commonly used in the context of classical (continuous and mixed) finite element methods, see \cite{boffi2013mixed}, or e.g. hybridized discontinuous Galerkin methods, see \cite{MR3585789}.
The main idea is to eliminate degrees of freedom by considering local problems and to assemble a global system of equations only for the remaining degrees of freedom.

Classically, this is done on a linear algebra level using a Schur complement strategy.
Ordering unknowns by interior (i) and skeleton/interface (s) degrees of freedom yields
\begin{subequations}
    \label{eq:staticcond:linalg}
\begin{equation}
\left(
    \begin{matrix}
A_{ii} & A_{is} \\
A_{si} & A_{ss}
\end{matrix}
\right)
\cdot
\left(
\begin{matrix}
u_i \\ u_s
\end{matrix}
\right)
=
\left(
\begin{matrix}
f_i \\ f_s
\end{matrix}
\right).
\label{eq:staticcond:linalg1}
\end{equation}
Here, $A_{ii}$ is block diagonal and cheap to invert so that we can express $u_i = u_i^f + u_i^h$ with $u_i^f = A_{ii}^{-1} f_i$ the load-driven bubble part and $u_i^h = - A_{ii}^{-1} A_{is} u_s$ the suitable discrete \emph{harmonic extension} of the skeleton unknowns to the interior. Plugging this identity into the equation for $u_s$ results in the Schur complement equation
\begin{equation}
\label{eq:staticcond:linalg2}
S\,u_s = g,\quad \text{ with }\quad S := A_{ss} - A_{si}A_{ii}^{-1}A_{is} \quad \text{ and } \quad g := f_s - A_{si}A_{ii}^{-1}f_i.
\end{equation}
\end{subequations}
After the solution for the skeleton unknowns $u_s$, the interior unknowns can then be recovered by solving the independent local problems for the harmonic extension and the load-driven bubble part.

We now present an algebraically equivalent approach that reflects the local–global decomposition by means of the local space $\IL_h$ and the global space $\IT_h$.
The global degrees of freedom span a Trefftz space $\IT_h$, namely the space of \emph{harmonic extensions} of the skeleton functions.

To demonstrate this in more detail we consider a classical continuous finite element method applied to a PDE with a self-adjoint differential operator. If we take $\Vh$ to be the space of continuous finite element functions, then the space allows the splitting $ \Vh = \IL_h \oplus \IT_h$.
Here, the global space is given by 
\begin{equation}
    \IT_h = \{ v_h \in \Vh \text{ s.t. } \langle A_K v_h, q_h \rangle_K = 0\quad \forall q_h \in \QKh,\ K\in\Th \} ,
\end{equation}
where $\QKh= \{ q_h \in \Vh \text{ s.t. } q_h(x) = 0 \text{ for all } x \notin K\}$ are the bubbles and $A_K$ is the (local) differential operator of the PDE.
Note that $\IT_h$ is the orthogonal complement of the sum of the local spaces $\IL_h = \bigoplus_{K}\QKh$ with respect to the scalar product induced by the PDE operator. 

This splitting allows one to solve for $u_h=u_\IL + u_\IT$, with $u_\IL \in \IL_h$ and $u_\IT \in \IT_h$.
The local component $u_\IL$ is obtained by solving the independent local problems
\begin{align*}
    \langle A_K u_\IL, q_h \rangle_K = \langle \ell_K, q_h \rangle_K \quad \forall q_h \in \QKh,
\end{align*}
on each element $K\in\Th$, where $\ell_K$ is the restriction of the right-hand side to the element $K$.
The global component $u_\IT$ is then obtained by solving a global problem on $\IT_h$; since the basis functions of $\IT_h$ already incorporate the discrete harmonic extension from the skeleton, no separate extension step is required after the global solve.

In terms of the linear algebra objects in \eqref{eq:staticcond:linalg}, if we reinterpret the block structure in the orthogonal decomposition basis (i.e., identify $i\equiv \IL_h$ and $s\equiv \IT_h$ after a change of basis), the orthogonality yields $A_{is} = A_{si} = 0$ with a globally coupled matrix $A_{ss}$ and a block diagonal matrix $A_{ii}$. In other words: due to the orthogonality of the subspaces, the discrete harmonic extension from global/Trefftz/skeleton unknowns to interior bubbles becomes trivial ($u_i^h=0$), while at the same time the global problem no longer depends on $u_i$ (so $g=f_s$ in \eqref{eq:staticcond:linalg2}). The Schur complement $S$ is then precisely the stiffness matrix restricted to $\IT_h$,  $A_{ss}$.

\subsection{Towards a local--global framework}
All Trefftz-type discretizations discussed above can be interpreted within a common local--global decomposition. Concretely, we split the discrete trial space as
\[
\Vh = \IL_h \oplus \IT_h,
\]
with \emph{local} components $\IL_h$ and \emph{global}/Trefftz components $\IT_h$. The resulting discrete system can then be written in block form
\begin{equation}\label{eq:block-lg}
\begin{pmatrix}
A_{\IL\IL} & A_{\IL\IT} \\
A_{\IT\IL} & A_{\IT\IT}
\end{pmatrix}
\begin{pmatrix}
u_\IL\\ u_\IT
\end{pmatrix}
=
\begin{pmatrix}
\ell_\IL\\ \ell_\IT
\end{pmatrix},
\end{equation}
in analogy with the interface/interior partitioning in static condensation, cf.~\eqref{eq:staticcond:linalg1}.
Different Trefftz-type methods correspond to particular structural properties of the blocks in \eqref{eq:block-lg}:
\begin{itemize}
    \item Classical polynomial Trefftz DG selects $\IT_h$ so that all elementwise PDE constraints are exactly satisfied in the trial space and is typically formulated for homogeneous PDEs. In this case the local equation vanishes and one has $u_\IL=0$; the global problem is solved solely in $\IT_h$.
    \item Embedded Trefftz DG and quasi-Trefftz DG admit non-homogeneous problems through elementwise local solves in $\IL_h$ (particular solutions/affine shifts). By construction of $\IT_h$ via weak or Taylor–based Trefftz conditions, the coupling from global to local vanishes,
    \[
    A_{\IL\IT}=0,
    \]
    so that local problems are independent and can be solved in parallel, followed by a global solve restricted to $\IT_h$. In general, however, the converse coupling block $A_{\IT\IL}$ does not vanish automatically, even for self-adjoint PDE operators: unlike the pure static condensation setting, the discrete operators imposed on $\IL_h$ and on $\IT_h$ are not identical, and the resulting block structure need not be symmetric.
\end{itemize}

This viewpoint motivates a unified analysis in the subsequent sections, where we formalize local--global discretizations with an even more general structure that especially includes all the cases considered above.

\section{The framework for local-global discretizations} \label{sec:framework}
In this section we introduce the general framework for discretizations with a local-global decomposition.
We start with a generic definition of a well-posed continuous problem in \cref{ssec:continuousproblem}.
Then we set up notation and function spaces of a generic discretization in \cref{ssec:discretization}.
This is based on local subproblems, introduced in \cref{ssec:localprob}, a global problem on the remainder space, discussed in \cref{ssec:globalprob}, and the coupling of these problems \cref{ssec:globallocalprob}.
We present the a-priori error analysis of a local-global discretization in \cref{ssec:globallocalprobanal} and derive the main result of this section, \cref{thm:Tcoercivity:coupled}, which is the stability of the coupled system.
Finally, we discuss error estimates in weaker norms in \cref{ssec:weakernorms}.

\subsection{Continuous problem} \label{ssec:continuousproblem}
Let $\V$ and $\W$ be Hilbert spaces. We typically think about a Sobolev space on an open bounded Lipschitz domain $\Omega\subset \IR^d$ with $d=2,3$ (where a PDE problem may be posed). 
We consider the following abstract problem: Find $u\in \V$ such that
\begin{equation}\label{eq:abstract} 
    \addtocounter{equation}{1} \tag{$\texttt{PDE}|\theequation$}
    a(u,v) = \ell(v) \quad \forall v\in \W,
\end{equation}
where $a:\V\times \W\to\IR$ is a continuous bilinear form and $\ell\in \W'$. 
We assume that the problem is \hypertarget{def:tc}{$\tc$-coercive for some bounded bijective linear operator $\tc:\V\to \W$ such that for all $u\in \V$}
\begin{equation}\label{eq:coercivity}
    \addtocounter{equation}{1} \tag{$\tc\!|\theequation$}
    a(u,\tc u) \geq \norm{u}_\V^2.
\end{equation}
We note that the existence of such a $\tc$ operator is equivalent to the usual inf-sup condition for stability, cf. \cite[Thm. 1]{Ciarlet2012}. 
Especially, if $\V=\W$ and $a(\cdot,\cdot)$ is coercive, then $\tc$ is a scalar $\tc \in (0,\infty)$ and we have the usual coercivity property  $ a(u,u) \geq \frac1\tc \norm{u}_\V^2$.

\subsection{Underlying discrete spaces and other notation} \label{ssec:discretization}

As a starting point for the discussion we consider a family of index sets $\Th$, and Hilbert spaces $\Vh$ and $\Wh$, indexed by $h\in\calH$.
We will use the notation $\lesssim$ to denote inequalities up to a constant that is independent of the choice of $(\Vh, \Wh)$ in this family. 
These (finite dimensional) Hilbert spaces $\Vh$ and $\Wh$ define a generic (underlying) discretization of the problem \eqref{eq:abstract}.
We further assume that $\Vh$ is equipped with a suitable discrete norm $\norm{\cdot}_{\Vh}$.
We assume that $a_h(\cdot,\cdot)$ is defined on $\Vsh \times \Wh$ where \hypertarget{def:Vsh}{$\Vsh := \Vh + \V$} and $\Wh$ are equipped with suitable norms $\norm{\cdot}_{\Vsh}$ and $\norm{\cdot}_{\Wh}$, respectively.
We assume that $\norm{\cdot}_{\Vh}$ is also defined on $\Vsh$ and weaker than $\TnormDG{\cdot}$, i.e. $\norm{\cdot}_{\Vh} \lesssim \TnormDG{\cdot}$.

\subsection{Local subproblems}\label{ssec:localprob}

We consider a subspace \hypertarget{def:Uh}{ $\Uh \subseteq \Vh$ which can be decomposed into a disjoint sum of spaces 
\begin{align} \label{eq::Vhsplitting}
    \Uh := \bigoplus_{K\in\Th} \UKh \subseteq \Vh.
    \addtocounter{equation}{1} \tag{$\Uhnl|\theequation$}
\end{align}
}

Further, we consider a set of \hypertarget{def:AK}{linear maps representing (scaled) local versions of the operator $a(\cdot,\cdot)$ from \eqref{eq:abstract}, e.g. the differential operator in a PDE, 
\begin{equation}
    \addtocounter{equation}{1} \tag{$A_K|\theequation$}
\AK: \Vh \to \QK'\quad\text{for each $K\in\Th$},
\end{equation}
where $\QK$ is a suitable space.} We want to emphasize that although $\AK$ maps into the local space $\QK'$, it is defined on the whole space $\Vh$. 
Now let $\Vert \cdot \Vert_{\QK'}$ denote the usual dual norm on $\QK$, i.e. $\Vert \cdot \Vert_{\QK'} = \sup_{q\in \QK, \Vert q \Vert_{\QK} = 1} \langle \cdot, q \rangle$.
For the simultaneous application of $\AK,~K\in\Th$, to an element in $\Uh$ we introduce the notation 
\begin{equation}
    \hypertarget{def:ATh}{\ATh : \Uh \to \QTp} \text{ with }\ATh u_h = (\AK u_h)_{K\in\Th},
    \addtocounter{equation}{1} \tag{$A_{\Th}\!|\theequation$}
\end{equation}
with \hypertarget{def:QT}{$\QT := \Pi_{K\in\Th} \QK$} and \hypertarget{def:QTp}{$\QTp$} its dual.
The maps $\AK$ are assumed to be locally stable in the following sense.

\begin{assumption}[Simultaneous local stability]\label{ass:local}
There exist spaces $\QKh \subset \QK$ such that the linear maps $\AK$ restricted to $\UKh$, i.e.  $\AK : \UKh \to \QKh'$ define bijective maps and further there holds
\begin{equation}\label{eq:local_solveability}
    \norm{\ATh  u_h }_{\QThp}^2 = \sum_{K\in\Th} \norm{\AK  u_h }_{\QKh'}^2 \gtrsim \norm{u_h}_{\Vh}^2 \quad \forall u_h \in \Uh,
    \addtocounter{equation}{1} \tag{${A_{\Th}}{\texttt{-stab}}|\theequation$}
\end{equation}
with \hypertarget{def:QTh}{$\QTh := \Pi_{K\in\Th} \QKh$} and \hypertarget{def:QThp}{$\Vert q_h \Vert_{\QTh} = \big(\sum_K \Vert q_h|_K \Vert_{\QK}^2\big)^{1/2},\ q_h\in\QKh$}.
\end{assumption}

We further require continuity of the maps $\AK $ on the whole space $\Vh$ in a suitable norm.
\begin{assumption}[Simultaneous $\V_{\!\star \hspace*{-0.2mm}h}$-continuity]\label{ass:local_strong}
We assume
simultaneous continuity of $\ATh$ on $\Vsh$, i.e.
\begin{equation}\label{eq:triple_continuity_h}
    \norm{\ATh u }_{\QThp} \lesssim \TnormDG{u} \quad \forall u\in \Vsh.
    \addtocounter{equation}{1} \tag{${A_{\Th}}{\texttt{-cont}^{\!\ast}}|\theequation$}
\end{equation}
\end{assumption}
We note that if $\Th$ corresponds to an overlapping domain decomposition, \cref{ass:local_strong} reads  essentially as an assumption of a finite overlap of subdomains. 

The two previous assumptions together render the well-posedness of the following local subproblem(s):
Find $u_\IL \in \Uh$ so that with 
$\ell_{\Th}(\cdot) = (\ell_K(\cdot)_K)_{K\in\Th}$ for some suitable functio\-nals $\ell_K(\cdot)$ there holds 
\begin{equation} 
    \langle \ATh u_\IL, q_h \rangle = \ell_{\Th}(q_h) \quad \forall q_h\in \QTh.    \addtocounter{equation}{1} 
    \tag{$\texttt{loc}|$\theequation}
    \label{eq:global-onT}
\end{equation}
Note that in general \eqref{eq:global-onT} does not decompose into a set of local subproblems.
However, in a large class of discretization methods, cf. \cref{ssec:DGsetting} this is the case and \eqref{eq:global-onT} can be solved embarrassingly parallel.

Motivated by the previous two assumptions, we introduce the following semi-norm on $\Vsh$:
\begin{equation} \label{eq:AThnorm}
\AThnorm{v}{} := \norm{\ATh v}_{\QThp}\quad \forall v \in \Vsh,
\addtocounter{equation}{1} \tag{$|{\cdot}|_{A_{\Th}}|\theequation$}
\end{equation}
and note that under \cref{ass:local,ass:local_strong} $| \cdot |_{A_\Th}$ defines a norm on $\Uh$.

For the analysis in the later subsection we rely on both the previous assumptions. 
To streamline the verification of the assumptions, especially \cref{ass:local}, for many cases, we will provide sufficient conditions in \cref{sec:tools}.

\subsection{Global problem} \label{ssec:globalprob}
We define \hypertarget{def:ITh}{$\ITh$ as a complementary space of $\Uh$ in $\Vh$, thus we have the decomposition $\Vh = \ITh \oplus \Uh$.}
This allows us to formulate a global problem. 
Find $u_{\IT} \in \IT_h$ such that
\begin{align}
    a_h(u_{\IT},v_{\IT}) =  \ell_h(v_{\IT})  \quad \forall v_{\IT}\in \tch \ITh,
    \addtocounter{equation}{1} 
     \tag{$\texttt{glob}|$\theequation}
    \label{eq:global-onIT}
\end{align}
for some suitable discrete bilinear form $a_h(\cdot,\cdot)$, a discrete functional $\ell_h(\cdot)$, and a linear operator $\tch:\ITh\to\Wh$. 
We note that the bilinear form $a_h(\cdot,\cdot)$ is defined on $\Vsh \times \tch \ITh$ and not on $\Vh \times \Wh$.
Naturally, an underlying discretizations on $\Vsh\times\Wh$ can simply be restricted to $\Vsh \times \tch \ITh$.

For the stability of \eqref{eq:global-onIT} we require the following assumption.
\begin{assumption}\label{ass:assumption_Th}
There exists a uniformly bounded injective linear operator $\tch:\ITh \to \Wh$ such that for all $u_h\in \ITh$
    \begin{equation}\label{eq:coercivityTh}
        \addtocounter{equation}{1} \tag{$\IT_h\texttt{-stab}|$\theequation}
        a_h(u_h,\tch u_h) \geq \norm{u_h}_{\Vh}^2 
        \text{ with } \norm{\tch u_h}_{\Wh}\lesssim \norm{u_h}_{\Vh}.
    \end{equation}
\end{assumption}

The condition \eqref{eq:coercivityTh} poses the discrete version of the $\tc$-coercivity condition \eqref{eq:coercivity} and is equivalent to a discrete version of the inf-sup condition, see e.g. \cite[Thm. 2]{Ciarlet2012}.
For coercive problems the operator $\tch$ in \eqref{eq:global-onIT} is simply the scaled identity operator.

With the help of the $\tch$ operator in \cref{ass:assumption_Th}
we can also define an energy norm on $\ITh$ by
\begin{equation}\label{eq:ahnorm}
    \ahnorm{u_{\IT}}{} := \big(a_h(u_{\IT}, \tch u_\IT)\big)^{1/2} ~~\big( \geq \norm{u_{\IT}}_{\Vh} \big), \quad u_{\IT} \in \ITh.
    \addtocounter{equation}{1} \tag{$\Vert\cdot\Vert_{a_h}|$\theequation}
\end{equation}

We further assume that $a_h(\cdot,\cdot)$ is continuous in the following sense. 
\begin{assumption}\label{ass:cont_Th}
We assume that $a_h(\cdot,\cdot)$ is defined on $\Vsh \times \tch \ITh$ and continuous in the sense that 
\begin{equation}\label{eq:continuity_h}
    \addtocounter{equation}{1} \tag{$\IT_h\texttt{-cont}|$\theequation}
    a_h(u,v_h) \lesssim \TnormDG{u} \norm{v_h}_{\Wh} \quad \forall u\in \Vsh, \forall v_h\in \tch \ITh.
\end{equation}
\end{assumption}

We highlight that \cref{ass:assumption_Th,ass:cont_Th} represent a relaxation of the standard assumptions of $\tch$-coercivity and continuity.
Continuity is only require for test functions in $\tch \ITh$ as compared to the usual assumption for all test functions in $\Wh$.
The $\tch$-coercivity is only required for functions in $\ITh$. 

If the discrete formulation is derived by restricting a discretization that is well-posed on $\Vh \times \Wh$, we can inherit its $\tch$-coercivity and continuity.
In the special case that $\tch\ITh = \ITh$, we immediately inherit the $\ITh$-coericvity of the underlying discretization.

Otherwise, we must use the test space $\tch \ITh \subset \Wh$ to inherit the desired property.
The space $\tch \ITh$ might be tedious to construct. 
Therefore, in practice, it might be convenient to prove \cref{ass:assumption_Th} new, for a suitable $\tch$ operator that satisfies $\tch\ITh=\ITh$. 
An example of this can be found in \cite{LLS_NM_2024}, where the inf-sup condition for the Stokes problem is proved specifically on the Trefftz space.

\subsection{Coupled local-global problem} \label{ssec:globallocalprob}
In general the local and the global discretization problems may not fully decouple, but the global problem will depend on the local problems and vice versa, i.e. $\tilde \ell_K$ in \eqref{eq:global-onT} will depend on the global solution and $\tilde \ell_h$ in \eqref{eq:global-onIT} will depend on the local solution.
In this subsection we will discuss the coupling of these problems.

With the splitting of $\Vh$ into local and global spaces, we split the solution to the overall problem as 
$u_h = u_\IL + u_{\IT} \in \Uh \oplus \ITh$. 
We then define the overall discretization as the following coupled system:
\begin{align} \label{eq:block}
    \left(
    \begin{array}{c@{\quad\quad}c}
         \langle \ATh  \cdot, q_h \rangle &  \langle \ATh  \cdot, q_h \rangle  \\
        a_h(\cdot,v_h) & a_h(\cdot,v_h) 
    \end{array}
    \right) 
    \left(
     \begin{array}{c}
        u_\IL  \\
        u_{\IT}
     \end{array}
     \right)
    = \left(
    \begin{array}{c}
        \ell_{\Th}(q_h) \\
        \ell_h(v_h)
    \end{array}
    \right), 
\addtocounter{equation}{1} \tag{$\texttt{block}|$\theequation}
\end{align}
for all $q_h \in \QTh$ and $v_h \in \tch \ITh$.

We want to emphasize, that the test functions for the second row in \eqref{eq:block} are in the image of the $\tch$ operator (restricted to $\ITh$), which implies invertibility of the lower right block in the block system, assuming \cref{ass:assumption_Th}. Similarly, \cref{ass:local} assures invertibility of the upper left block in the block system.
We still need a suitable assumption on the off-diagonal blocks to ensure that the coupled system \eqref{eq:block} is solvable. 
This motivates the following assumption.
\begin{assumption}[Weak coupling]\label{ass:inexacttrefftz}
    We assume that the space $\ITh$, respectively $\Uh$, is chosen in such a way that there exists a constant $\rho\in [0,C_\IT^{-1})$ such that
    \begin{equation}\label{eq:inexacttrefftz}
        \AThnorm{u_{\IT}}{} 
        \le  \rho 
        \ahnorm{u_{\IT}}{} \quad \forall u_{\IT}\in \ITh,
    \quad \text{ where }\quad 
        C_\IT := \!\!\!\sup_{\substack{u_\IL\in \Uh, v_{\IT}\in \ITh}} \!\!\frac{|a_h(u_\IL,\tch v_{\IT})|}{ |u_\IL|_{\ATh} 
        \ahnorm{v_{\IT}}{}}.
        \addtocounter{equation}{1} \tag{$\rho, C_\IT|$\theequation}
    \end{equation}
\end{assumption}
A special, but practically very important, case is the case $\rho = 0$, i.e. where the upper right block vanishes.
Several Trefftz DG methods share this property, e.g. the Trefftz methods discussed in \cref{sec:Tlike}.

\subsection{A-priori error analysis of the coupled local-global problem} \label{ssec:globallocalprobanal}

For the analysis of the coupled problem \eqref{eq:block}, we define for $u_h \in \Vh$ the bilinear form
\hypertarget{def:Bh}{
\begin{equation}\label{eq:Bh}
    \Bh (u_h, z_h ) := \langle \ATh u_h, q_h \rangle + a_h(u_h,v_h)= \sum_{K\in \Th} \langle \AK  u_h, q_K \rangle + a_h(u_h,v_h),
    \addtocounter{equation}{1} \tag{$B_h|$\theequation}
\end{equation}
}
where $z_h = (q_h, v_h) \in \hypertarget{def:Wth}{\Wth := \QTh \times \tch \ITh}$. 
Using the unique decomposition $u_h = u_\IL + u_{\IT}$, equation 
$\Bh (u_h, z_h ) = \Bh (u_\IL + u_{\IT}, (q_h, v_h))$ reflects the structure of the block system~\eqref{eq:block}. 

Solving the coupled problem \eqref{eq:block} is equivalent to the following problem: Find $u_h\in \Vh$ such that
\begin{equation}\label{eq:Bh-sys}
    \Bh(u_h, z_h) = \ell_h(v_h) + \ell_{\Th}(q_h) \quad \forall z_h = (q_h, v_h) \in \Wth.
    \addtocounter{equation}{1} \tag{$B_h\texttt{-sys}|$\theequation}
\end{equation}

\begin{remark}\label{rem:indepL}
It is crucial to note that problem \eqref{eq:Bh-sys} only depends on the choice $\QTh$ and $\tch \ITh$ and not on the choice of $\Uh$.
For any choice of $\Uh$ such that $\Vh=\Uh \oplus \ITh$ $u_h=u_\IL + u_{\IT}$ solves \eqref{eq:Bh-sys} if and only if $u_\IL \in \Uh$ and $u_{\IT} \in \ITh$ solve the coupled system \eqref{eq:block}.
In particular, solving \eqref{eq:block} for two different choices $\Uh$ and $\widetilde{\Uh}$ will ultimately yield the same solution $u_h=u_\IL + u_{\IT}= \tilde u_{\widetilde{\IL}} + \tilde u_{\IT}$ of \eqref{eq:Bh-sys}.
\end{remark}

The solvability of \eqref{eq:Bh-sys} is discussed next.
\begin{theorem} \label{thm:Tcoercivity:coupled}
    Assume that \cref{ass:local,ass:assumption_Th,ass:inexacttrefftz} hold.
    Then, there exist maps 
    $\tc_{\Th}:$ $\Vh \to \QTh$, 
    and 
    $\tc_{\ITh}:$ $\Vh \to \ITh$, 
    so that $\Bh(\cdot,\cdot)$ 
     is $\tc$-coercive on $\Vh\times \Wth$
    for the $\tc$-coercivity map $\tc_{\Vh}:$ $\Vh \to \Wth$, $u_h \mapsto  (\tc_{\Th} u_h , \tc_{\ITh} u_h )$, i.e. 
    for all $u_h\in \Vh$
    \begin{equation}\label{eq:Bhstab}
    \Bh(u_h, \tc_{\Vh} \! u_h) = \langle \ATh  u_h, \!\tc_{\Th}\! u_h \rangle + a_h(u_h, \!\tc_{\ITh}\! u_h )\gtrsim 
    \Bhnorm{u_h}{2}
    \left(\gtrsim \normDG{u_h}^2\right),
    \addtocounter{equation}{1} \tag{$B_h\texttt{-stab}|$\theequation}
    \end{equation}
    \begin{equation}\label{eq:Bhnorm}
        \text{where }\quad \Bhnorm{u_h}{2} := \AThnorm{u_\IL}{2} + \ahnorm{u_{\IT}}{2} \text{ for } u_h \in \Vh,~ u_\IL \in \Uh, u_{\IT} \in \ITh,
        \addtocounter{equation}{1} \tag{$\Vert \cdot \Vert_{B_h}|$\theequation}
    \end{equation}
    where we made use of the unique decomposition $u_h = u_{\IT} + u_\IL \in \ITh \oplus \Uh = \Vh$.
    Moreover, the maps $\tc_{\Th}$ and $\tc_{\ITh}$ and hence $\tc_{\Vh}$ are continuous w.r.t. the $\Bhnorm{\cdot}{}$-norm, i.e.:
    \begin{equation}\label{eq:Bhstab_Tbound}
        \norm{\tc_\Th u_h }_{\QTh} + \norm{\tc_{\ITh} u_h }_{\Wh} \lesssim 
        \Bhnorm{u_h}{} 
        \quad \forall u_h \in \Vh.
        \addtocounter{equation}{1} \tag{$T_{\Vh}\texttt{-cont}|$\theequation}
    \end{equation}
\end{theorem}
\begin{proof}
    We define $\tc_{\Th}$ and $\tc_{\ITh}$, starting with $\tc_{\Th}$ through its components $\tc_K$, $K\in\Th$: 
    \begin{align}
    \tc_K u_h &:=  R_K \AK (u_\IL - u_{\IT}) \in \QKh.
    \addtocounter{equation}{1} \tag{$\tc_{K}\!|$\theequation}
    \end{align}
    Here, $R_K: \QKh' \to \QKh$ denotes the Riesz operator in $\QKh'$  so that
    $\langle q_h, R_K p_h \rangle = (q_h,p_h)_{\QKh'}$ for $p_h,q_h \in \QKh'$ for $K\in\Th$. 
    With the definition of norms  \eqref{eq:AThnorm}, a triangle inequality and \eqref{eq:inexacttrefftz}, we obtain the following continuity estimate for $\tc_{\Th}$:
    \begin{align} \label{eq:tcTh-cont}
    \norm{\tc_{\Th} u_h}_{\QTh} & \!\lesssim\! \AThnorm{u_\IL - u_{\IT}}{} \!\!\leq \AThnorm{u_\IL}{} \!\!+\! \AThnorm{u_{\IT}}{}\!\stackrel{}{\lesssim}\! \AThnorm{u_\IL}{} \!\!+\! \ahnorm{u_{\IT}}{} ~
    \forall u_h \in \Vh.
    \addtocounter{equation}{1} \tag{$\tc_{\Th}\!\!\texttt{-cont}|$\theequation}
    \end{align}
    Next, for a constant $\beta_\rho  > 0$ to be chosen later (in dependence of $\rho$ and $C_\IT$) we define $\tc_{\ITh}$ as
    \begin{align}
        \tc_{\ITh} u_h &:= \beta_\rho  \tch (u_{\IT} - P_{\IT} u_\IL) \in \tch \ITh,
        \addtocounter{equation}{1} \tag{$\tc_{\IT_h}\!|$\theequation}
    \end{align}
     where 
    $P_\IT: \Uh \!\to\! \ITh$ is the solution map $P_{\IT}\!:\! u_\IL \!\mapsto\! v_{\IT} \!\in\! \ITh$ to the problem: Find $v_{\IT}\!\in\!\ITh$ s.t.
    \begin{equation}
    a_h(w_{\IT}, \tch v_{\IT}) = a_h(u_\IL, \tch w_{\IT})\quad \forall w_{\IT} \in \ITh.
    \addtocounter{equation}{1} \tag{$P_{\IT}|$\theequation} \label{eq:PIT}
    \end{equation}
    We note that \eqref{eq:PIT} is well-posed due to the coercivity of $a_h(\cdot,\tch \cdot)$ on $\ITh$, cf. \cref{ass:assumption_Th}.
    
    Next, we check for the continuity of $\tc_{\ITh}$ and begin with the continuity of $P_\IT$. 
    Taking $w_{\IT} = P_\IT u_\IL$ in \eqref{eq:PIT}, together with the definition of $C_\IT$  yields \vspace*{-1em}
    \begin{equation}
        |P_{\IT} u_\IL|_{a_h}^2 = a_h(P_\IT u_\IL, \tch P_\IT u_\IL ) \!\!\!\!\stackrel{\eqref{eq:PIT}}{=}\!\!\!\! a_h(u_\IL, \tch P_\IT u_\IL ) \!\!\!\!\!\!\stackrel{\eqref{eq:inexacttrefftz}}{\leq} \!\!\!\!\!\!C_\IT 
        |u_\IL|_{A_{\Th}} |P_{\IT} u_\IL|_{a_h}. \label{eq:ast}
        \addtocounter{equation}{1} \tag{$P_{\IT}\texttt{-cont}|$\theequation}
    \end{equation}
    Dividing by $\ahnorm{P_{\IT} u_\IL}{}$ yields $\ahnorm{P_{\IT} u_\IL}{} \leq C_{\IT} \AThnorm{u_\IL}{}$\!. 
    From continuity of $\tc_h$ (cf. \cref{ass:assumption_Th}), a triangle inequality and  \eqref{eq:ast} we obtain continuity for $\tc_{\ITh}$:
    \begin{equation} \label{eq:tcITh-cont}
        \norm{\tc_{\ITh} u_h}_{\Wh} \!\lesssim\! \norm{u_{\IT} \!- P_{\IT} u_\IL}_{\Vh} \!\leq\!  \ahnorm{u_{\IT}}{} \!\!+ \ahnorm{P_{\IT} u_\IL}{} \stackrel{}{\lesssim} \ahnorm{u_{\IT}}{} \!\!+ \AThnorm{u_{\IL}}{}
        ~ \forall u_h \in \Vh.
        \addtocounter{equation}{1} \tag{$\tc_{\IT_h}\!\!\texttt{-cont}|$\theequation}
    \end{equation}

    Having defined the $\tc_{\Vh}$-coercivity map through its two components, we prepare a bound for the contribution of $u_\IL$ in the $a_h(\cdot,\cdot)$ bilinear form. 
    Repeating the first two equalities (in opposite direction) in \eqref{eq:ast} and plugging in the resulting continuity bound for $P_{\IT}$ yields
    \begin{equation}
        a_h(u_\IL, \tch P_\IT u_\IL ) = a_h(P_\IT u_\IL, \tch P_\IT u_\IL ) \leq C_\IT^2 |u_\IL|_{A_{\Th}}^2.
        \addtocounter{equation}{1} \tag{$\ast$} \label{eq:astast}
    \end{equation}
    Finally, we can plug in the definitions of $\tc_{\Th}$ and $\tc_{\ITh}$ into the bilinear form $\Bh$ to obtain the desired coercivity estimate \eqref{eq:Bhstab} which we do in two steps corresponding to the two parts.
    Considering the $\tc_{\ITh}$-part of $\Bh(u_h,T_{\Vh}u_h)$ we obtain the estimate
    \begin{align*}
        \beta_\rho^{-1} a_h(u_h, \tc_{\ITh} u_h) &= a_h(u_\IT + u_\IL, \tch (u_\IT - P_{\IT} u_\IL))\\[-3.3ex]
                                                        &= a_h(u_\IT, \tch u_\IT) 
                                                        - a_h(u_\IL, \tch P_{\IT} u_\IL)
                                                        \overbrace{- a_h(u_\IT, \tch P_{\IT} u_\IL) 
                                                        + a_h(u_\IL, \tch u_\IT)}^{=0 \text{ with }\eqref{eq:PIT}} \\
                                                        &\stackrel{\eqref{eq:astast}}{\geq} |u_{\IT}|_{a_h}^2 - C_\IT^2 
                                                        |u_\IL|_{A_{\Th}}^2.
    \end{align*}
Now estimating the $\tc_\Th$ part of $\Bh(u_h,T_{\Vh}u_h)$ we obtain 
\begin{align*}
    \langle \ATh  u_h, \tc_\Th u_h \rangle &= \sum_{K\in \Th} \overbrace{\langle \AK  (u_\IL + u_{\IT}), R_K \AK  (u_\IL - u_\IT)\rangle }^{=(\AK  (u_\IL + u_{\IT}), \AK  (u_\IL - u_\IT))_{\QKh'}}
     \\
    & = \AThnorm{u_\IL}{2} - \AThnorm{u_\IT}{2} 
    \geq \AThnorm{u_\IL}{2} - \rho^2 \ahnorm{u_\IT}{2},
\end{align*}
where in the last step we used \eqref{eq:inexacttrefftz}.
Summing up the two estimates we obtain
\begin{align*}
    \Bh(u_h, \tc_{\Vh} u_h) &\geq \beta_\rho
     \ahnorm{u_\IT}{2} - \beta_\rho  C_\IT^2 
    \AThnorm{u_\IL}{2} + \AThnorm{u_\IL}{2} - \rho^2 \ahnorm{u_\IT}{2}
    \\
                            &\geq (\beta_\rho -\rho^2) \ahnorm{u_\IT}{2} + (1-C_\IT^2 \beta_\rho)  \AThnorm{u_\IL}{2}.
\end{align*}
Now choosing $\beta_\rho \in ( \rho^2, C_\IT^{-2})$, e.g. $\beta_\rho = \frac{C_\IT^{-2}+\rho^2}{2}$ for $C_\IT >0$ or $\beta_\rho = 2\rho^2$ for $C_\IT=0$, we obtain the desired coercivity.
\end{proof}

With \cref{thm:Tcoercivity:coupled} and additionally assuming that \cref{ass:local_strong,ass:cont_Th} holds, we can relate the norms $\norm{\cdot}_{\Vh}$,  $\norm{\cdot}_{\Vsh}$ and $\Bhnorm{\cdot}{}$ for functions in $u_h \in \Vh$:
\begin{align*}
\Vert u_h & \Vert_{\Vh} \!
\stackrel{\triangle}{\leq} \!
\norm{u_\IL}_{\Vh} +
\norm{u_{\IT}}_{\Vh}
\!\!\!\squeezedstackrel{\eqref{eq:local_solveability}, \eqref{eq:coercivityTh}}{\lesssim} \!\!\!
\AThnorm{u_\IL}{} + \ahnorm{u_\IT}{} 
\!\!\squeezedstackrel{\eqref{eq:Bhnorm}}{=}
 \Bhnorm{u_h}{}~~
\!\!\squeezedstackrel{\eqref{eq:Bhstab}}{\lesssim} \
\Bhnorm{u_h}{-1} \Bh(u_h, \tc_{\Vh} u_h) 
\\
&
\squeezedstackrel{\eqref{eq:Bh}}{=} 
\Bhnorm{u_h}{-1} \big( \langle \ATh  u_h, \!\tc_{\Th}\! u_h \rangle + a_h(u_h, \!\!\tc_{\ITh}\! u_h ) \big)
\!\!\!\!\squeezedstackrel{\eqref{eq:triple_continuity_h}, \eqref{eq:tcITh-cont}, \eqref{eq:tcTh-cont}}{\lesssim} \!\!\!
\Bhnorm{u_h}{-1} \norm{u_h}_{\Vsh} (\AThnorm{u_\IL}{}\!\! +\! \ahnorm{u_\IT}{} ) 
\!=\! \norm{u_h}_{\Vsh},
\end{align*}
i.e. in total we have 
\begin{align}  \label{eq:VhBhVhstar_normequi} 
    \norm{u_h}_{\Vh} \lesssim \Bhnorm{u_h}{} \lesssim \norm{u_h}_{\Vsh}.
\end{align}
Next, we exploit the stability result to obtain error bounds.
\begin{corollary}[Strang-type result]\label{cor:cea}
    Let $u\in V$ be the solution to \eqref{eq:abstract} and $u_h\in \Vh$ be the solution to \eqref{eq:block}. Assume that 
    \cref{ass:local,ass:local_strong,ass:assumption_Th,ass:inexacttrefftz,ass:cont_Th}
    hold. Then there holds the bound
    \begin{equation}\label{eq:global-approx}
        \norm{u-u_h}_{\Vh} \!\lesssim\! \inf_{v_h \in \Vh}\!\TnormDG{u - v_h} + \norm{a_h(u,\cdot) - \ell_h}_{\tch\ITh'} + \norm{A_\Th u - \ell_\Th}_{\QThp},
        \addtocounter{equation}{1} \tag{$B_h\texttt{-Str}|$\theequation}
    \end{equation}
    where the hidden constants depend only on the constants in the assumptions. 
\end{corollary}
\begin{proof}
    The solution $u_h$ of $\eqref{eq:block}$ solves $\eqref{eq:Bh-sys}$ and hence we can apply the previous theorem.
    Let $v_h \in \Vh$. Let $\tc_{\Vh}:$ $\Vh \to \Wth$, $u_h \mapsto  (\sum_K \tc_{K} u_h , \tc_{\ITh} u_h )$ be as in the previous theorem. By \eqref{eq:Bhstab} we have
    \begin{align*}
        \Bhnorm{v_h-u_h}{2}
        & \lesssim B_h( v_h-u_h, \tc_{\Vh} (v_h - u_h))\\
        & \lesssim B_h( v_h-u, \tc_{\Vh} (v_h - u_h))
        + B_h( u-u_h, \tc_{\Vh} (v_h - u_h))
        \\ &= \langle \ATh  (v_h - u), \tc_{\Th} (v_h - u_h) \rangle + a_h(v_h - u, \tc_{\ITh} (v_h- u_h) )
        \\ & \qquad \quad+ \langle \ATh  (u - u_h), \tc_{\Th} (v_h - u_h) \rangle + a_h(u - u_h, \tc_{\ITh} (v_h- u_h) )
        \\ & \lesssim
        \big( \AThnorm{v_h - u}{} + \norm{\ATh u - \ell_\Th }_{\QThp} \big) \norm{\tc_\Th(v_h - u_h)}_{\QTh} 
        \\ & \qquad \quad + \big( \norm{v_h - u}_{\Vsh} + \norm{a_h(u,\cdot) - \ell_h}_{\tch\ITh'} \big) \norm{\tc_{\ITh} (v_h- u_h)}_{\Wh},
    \end{align*}
    where we have used \cref{ass:cont_Th} in the third step.
    Using \cref{ass:local_strong} to bound $\AThnorm{v_h - u}{}$ by $\norm{v_h - u}_{\Vsh}$ and the continuity bound \eqref{eq:Bhstab_Tbound} we conclude 
    \begin{align*}
        \Bhnorm{v_h-u_h}{2}
        \lesssim &
        \big( \norm{v_h - u}_{\Vsh} 
         + \norm{a_h(u,\cdot) - \ell_h}_{\tch\ITh'} 
         + \norm{\ATh u - \ell_\Th }_{\QThp}  \big) \cdot 
         \Bhnorm{v_h-u_h}{},
    \end{align*}                
    Dividing by the latter factor we obtain
    \begin{align*}
        \Bhnorm{v_h-u_h}{}
        \lesssim &
        \norm{v_h - u}_{\Vsh} 
         + \norm{a_h(u,\cdot) - \ell_h}_{\tch\ITh'} 
         + \norm{\ATh u - \ell_\Th }_{\QThp}.
    \end{align*}                
    Now with  \eqref{eq:VhBhVhstar_normequi} we have $\norm{v_h - u_h}_{\Vh} \lesssim          \Bhnorm{v_h-u_h}{}$, 
    and thus bound \eqref{eq:global-approx} follows by the triangle inequality 
    $
    \norm{u_h - u}_{\Vh} 
    \leq
    \norm{u - v_h}_{\Vh}
    +
    \norm{u_h - v_h}_{\Vh},
    $
    bounding $\norm{u - v_h}_{\Vh} \lesssim \norm{u - v_h}_{\Vsh}$
    and finally taking the infimum over all $v_h\in \Vh$.

\end{proof}

\begin{remark}
    Let us stress that the error bound in \eqref{eq:global-approx} depends on an approximation bound over the whole space $\Vh$. This is a major feature of this analysis framework and is in contrast to the standard analyses on polynomial Trefftz methods, 
    with the exception of the work \cite{lozinski19} where a Trefftz method is proposed without identifying it as such.
    The usual analysis in the literature is on the global problem on the Trefftz space $\ITh$, 
    and employs the (averaged) Taylor polynomials as interpolation operators.
    The presented framework allows for a generic error analysis of Trefftz methods in terms of the approximation error in the whole space $\Vh$.
    Best approximation results from the underlying space $\Vh$ can be directly transferred to the Trefftz space $\ITh$. 
\end{remark}

\subsection{Aubin-Nitsche arguments for error estimates in weaker norms}\label{ssec:weakernorms}
In this section we discuss the possibility of obtaining error estimates in weaker norms associated to some Hilbert space\footnote{A famous example is $\IH = L^2(\Omega) \supset H^1(\Omega) = \V$ for elliptic PDEs with $L^2-H^2$ regularity, e.g. the Poisson problem on domains with smooth or convex boundaries.} $\IH\supset \V$. We make the usual assumption that $\V$ is dense in $\IH$ and the additional assumption that also $\Vsh \subset \IH$ and that $\normDG{\cdot}$ is stronger than the norm
$\norm{\cdot}_\IH$. 
As common, in order to derive the estimates in this section we consider the dual problem. For this we introduce the space $\Wsh:=W+\Wh$ with a suitable norm $\norm{\cdot}_{\Wsh}$. We assume that this norm is stronger than $\norm{\cdot}_W$ and $\norm{\cdot}_{\Wh}$. A crucial ingredient for the error bound is a suitable regularity condition of certain functions in $W$ for which the estimates \eqref{eq:H-reg-bound} and \eqref{eq:H-reg-bound2} below can be derived. More precisely, we want to
consider functions $z \in W$ such that $a(\cdot,z)\in \IH'$, i.e. $z$ such that 
\[
    \norm{a(\cdot,z)}_{\IH'} = \sup_{v\in \V} \frac{|a(v,z)|}{\norm{v}_{\IH}} < \infty. 
\]

Equipped with this regularity notion, we are able to formulate the following error bound in the $\IH$-norm. 

\begin{theorem} \label{th:aubinnitsche}
    Additionally to \cref{ass:local,ass:local_strong,ass:assumption_Th,ass:inexacttrefftz,ass:cont_Th}, we assume that the norms $\norm{\cdot}_{\Vh}$ and  $\TnormDG{\cdot}$ are equivalent on $\Vh$ and that $\TnormDG{\cdot}$ is stronger than $\norm{\cdot}_\V$ on $\V$. Furthermore, we assume the following consistency 
    \begin{equation}
        \label{eq:a-bit-consistent}
        a_h(v, z) = a(v,z) \quad \forall v \in \Vsh, z \in \{y\in W\mid a(\cdot,y)\in \IH' \},
        \addtocounter{equation}{1} \tag{$a_h\texttt{-adj\!.\!cons\!.}|$\theequation}
    \end{equation}
    as well as the following continuity for the consistent extension of $a_h$
    \begin{equation}
        \label{eq:a-bit-cont}
        a_h(v, z) \lesssim \TnormDG{v}\norm{z}_{\Wsh} \quad \forall v \in \Vsh, z \in \Wh + \{y\in W\mid a(\cdot,y)\in \IH' \}. 
        \addtocounter{equation}{1} \tag{$a_h\texttt{-adj\!.\!cont\!.}|$\theequation}
    \end{equation}
    Assume that there exists a constant $h_H>0$ such that for all $z\in W$ with $a(\cdot,z)\in \IH'$ 
    \begin{align}\label{eq:H-reg-bound}
        \inf_{z_h \in \tch \ITh} \norm{z -z_h}_{\Wh} &\lesssim
        h_H \sup_{v\in \V} \frac{|a(v,z)|}{\norm{v}_{\IH}},\\
        \label{eq:H-reg-bound2}
        \inf_{z_h \in \Wh} \norm{z -z_h}_{\Wsh} &\lesssim
        h_H \sup_{v\in \V} \frac{|a(v,z)|}{\norm{v}_{\IH}}.
    \end{align}
    Let $u\in \V$ be the solution to \eqref{eq:abstract} and $u_h\in \Vh$ be the solution to \eqref{eq:block}. Then, we have the approximation bound in the $\norm{\cdot}_H$-norm
\[
\norm{u-u_h}_H \lesssim h_H \inf_{v_h\in \Vh}\TnormDG{v_h-u} + (1+h_H)\norm{a_h(u,\cdot) - \ell_h}_{\tch\ITh'} + h_H \norm{\ATh u - \ell_\Th}_{\QThp}.
\]
\end{theorem}
\begin{proof} 
    We start by considering the consistency error. Let $e_h \in \ITh$ be the unique solution to the problem
    \[
    a_h(e_h,w_h) = a_h(u_h - u,w_h) \quad \forall w_h \in \tch\ITh,
    \]
    which exists and satisfies the bound $\norm{e_h}_H \lesssim \normDG{e_h} \lesssim \norm{a_h(u,\cdot) - \ell_h}_{\tch\ITh'}$ due to \cref{thm:Tcoercivity:coupled}. It remains to bound $e = u_h - u - e_h$ in the $\IH$-norm.  
 Due to the $\tc$-coercivity property \eqref{eq:coercivity} there exists a unique $y\in \V$ such that 
\[
a(v,Ty)= (e,v)_H\quad\text{for all }v\in \V \text{ with } \norm{y}_{\V}  \lesssim \norm{(e,\cdot)_H}_{\V'}\lesssim \norm{e}_H.
\]
Let $z:=\tc y$.
By construction we have that $a(\cdot, z) \in \IH'$, thus by \eqref{eq:a-bit-consistent} and the definition $e_h$, we have 
\begin{align*}
\norm{e}_H^2 & = a(e, z) =  a_h(u_h - u - e_h, z) 
\\ & =  a_h(u_h - u - e_h, z) - a_h(u_h - u - e_h, z_h) =  a_h(u_h - u - e_h, z - z_h)
\\ & =  a_h(u_h - u - e_h, w_h - z_h) +  a_h(u_h - u - e_h, z - w_h),
\end{align*}
for any $z_h \in \tch \ITh$ and any $w_h \in \Wh$. Using the continuity of $\bar a_h$, we obtain 
\begin{align*}
\norm{e}_H^2 &\lesssim \norm{e }_{\Vsh} ( \norm{w_h - z_h}_{\Wh} + \norm{z - w_h}_{\Wsh} )\\
& \lesssim \norm{e }_{\Vsh} (\norm{w_h - z}_{\Wh} + \norm{z-w_h}_{\Wsh}+ \norm{z-w_h}_{\Wsh}).
\end{align*}
Taking the infimum over all $z_h\in \tch \ITh$ and $w_h \in \Wh$ and using $\sup_{v\in \V} \frac{|a(v,z)|}{\norm{v}_{\IH}}= \norm{e}_H$ in \eqref{eq:H-reg-bound} and \eqref{eq:H-reg-bound2} we obtain 
\[ \norm{e}_H    \lesssim \norm{u_h - u - e_h}_{\Vsh} h_H .
\]
We now use $\TnormDG{\cdot}\lesssim \norm{\cdot }_{\Vh}$ on $\Vh$ and $\normDG{\cdot}\lesssim \TnormDG{\cdot }$ on $\Vsh$ to obtain
\begin{align*}
    \norm{u_h -u }_H & \lesssim \norm{e_h}_H  + \norm{e}_H  \lesssim (1+h_H) \TnormDG{e_h} +  h_H \TnormDG{u_h-u} 
    \\ & \lesssim (1+h_H) \normDG{e_h} + h_H \normDG{u_h-v_h} + h_H \TnormDG{u-v_h} 
    \\ & \lesssim (1+h_H) \normDG{e_h} +  h_H \normDG{u_h-u} +  h_H \normDG{u-v_h}+  h_H \TnormDG{u-v_h} 
    \\ & \lesssim (1+h_H) \norm{a_h(u,\cdot) - \ell_h}_{\tch\!\ITh'} +  h_H \normDG{u_h-u} +   h_H \TnormDG{u-v_h},
\end{align*}
for all $v_h\in \Vh$. Combining this with \eqref{eq:global-approx} and taking the infimum over all $v_h\in \Vh$ we obtain the desired bound.
\end{proof}

The bound \eqref{eq:H-reg-bound2} is common for estimates in weaker norms, however the additional assumption that \eqref{eq:H-reg-bound} holds, is also necessary. To obtain 
\eqref{eq:H-reg-bound} from \eqref{eq:H-reg-bound2} we need to assume additionally that $z$ is suitably approximated by $z_h = z_\IL + z_\IT$ for some $z_\IT \in \tch \ITh$ and $z_\IL$ which satisfies 
$\norm{z_\IL } \lesssim h_H \sup_{v\in \V} \frac{|a_h(v,z)|}{\norm{v}_{\IH}}$. 
In the special case that $\V=\W$, $\tch \ITh = \ITh$ and that \cref{ass:inexacttrefftz} holds with $\rho= 0$ the following Lemma provides a simple way to bound $h_H$ in \eqref{eq:H-reg-bound}.

\begin{lemma}\label{lemma:H-reg-bound}
    Assume that \cref{ass:local,ass:local_strong,ass:assumption_Th,ass:cont_Th}
    hold. Assume that \cref{ass:inexacttrefftz} holds with $\rho=0$ and that $\V=\W$.

    If $h_H >0$ is chosen such that for all $z \in \V$ with $a(\cdot,z) \in \IH'$
    \begin{equation} \label{eq:H-reg-bound-weak}
        \inf_{z_h \in \Vh} \norm{z - z_h}_{\Vsh} \lesssim h_H \sup_{v\in \V} \frac{|a(v,z)|}{\norm{v}_{\IH}} \quad \text{ and } \quad \AThnorm{z}{}  \lesssim h_H \sup_{v \in \V} \frac{|a(v,z)|}{\norm{v}_\IH},
        \addtocounter{equation}{1} \tag{$h_H\texttt{-weak}|$\theequation}
    \end{equation}
    then 
    \begin{equation}\label{eq:A-h-reg}
        \inf_{z_h \in  \ITh} \norm{z -z_h}_{\Vh} \lesssim 
         h_H \sup_{v\in \V} \frac{|a(v,z)|}{\norm{v}_{\IH}}.
    \end{equation}
\end{lemma}
\begin{proof}
    Let $z_h \in \Vh$ be the solution of \eqref{eq:block} with right-hand side given by
    $$\left(
        \begin{array}{c}
            \ell_{\Th}(\cdot ) \\
            \ell_h( \cdot )
        \end{array}
        \right) = \left(
            \begin{array}{c}
                0 \\
                a(z,\cdot)
            \end{array}
            \right).
        $$
        The bound \eqref{eq:global-approx} shows that 
        \[
            \normDG{z-z_h} \lesssim \inf_{v_h\in \Vh }\TnormDG{z - v_h} + \AThnorm{z}{}.
        \]
        Using \eqref{eq:H-reg-bound-weak}, we obtain
        \[
        \normDG{z-z_h} \lesssim h_H \sup_{v\in \V} \frac{|a(v,z)|}{\norm{v}_{\IH}}.
        \]
        Noting that due to $\rho=0$ in \cref{ass:inexacttrefftz} we have that $z_h \in \ker \ATh = \ITh$ completes the proof. 
\end{proof}

\section{Tools to verify the framework assumptions} \label{sec:tools}

The abstract framework in \cref{sec:framework} relies on several assumptions. Most of them are quite natural and easy to check for Trefftz-like methods that are based on an underlying standard discretization. Less common and obvious is perhaps \cref{ass:local}. 
In this section we hence provide tools that make the analysis framework more accessible by providing alternative sufficient conditions for some of the previous assumptions. 

In \cref{ssec:suffa1} we give sufficient conditions of the stability of the operator $\ATh$.
To prove local stability we introduce the helpful notion of a prototype operator in \cref{sec:protop}.
In \cref{ssec:DGsetting} we consider discontinuous Galerkin discretizations, which allows us to simplify several assumptions due to its local, element-wise, approach.
In \cref{ssec:genericembt} we discuss a generic recipe for the embedded Trefftz method, that guarantees local stability.

\subsection{Sufficient conditions for \texorpdfstring{\cref{ass:local}}{Assumption 1}}\label{ssec:suffa1}

From \cref{rem:indepL} it is clear that it suffices to verify \cref{ass:local} for a convenient choice of $\Uh$.
To simplify this further, we can also formulate the following alternative to \cref{ass:local}
 which is stronger, but may be easier to check in many cases:
\begin{lemma}\label{lemma:quasiortho}
    If there exists a family of projections $P_K: \Uh \to \UKh$ 
    satisfying
    \begin{subequations}
    \begin{alignat}{3}
    \AK  &= \AK  P_K \text{ in } \QKh' && \quad  \forall K\in\Th,  \label{eq:quasiortho:a}\\
    \textstyle{\sum_{K}} \big\Vert P_K u \Vert_{\Vh}^2 
        &\gtrsim \norm{u}_{\Vh}^2
        &&\quad\forall u \in \Uh,
    \label{eq:quasiortho:b}
    \\
    c_{\text{\tiny\eqref{eq:quasiortho:c}}} \norm{u}_{\Vh}  &\le \norm{\AK  u}_{\QKh'} &&\quad  \forall K\in \Th, \forall u\in \UKh, &&
    \label{eq:quasiortho:c}
    \end{alignat}
\end{subequations}
for some $c_{\text{\tiny\eqref{eq:quasiortho:c}}}$,
    then \cref{ass:local} holds.
\end{lemma}
\begin{proof}
    Let $u \in \Uh$, using, in order, \eqref{eq:quasiortho:b}, \eqref{eq:quasiortho:c}, and then \eqref{eq:quasiortho:a} we have
    \begin{align*}
        \norm{u}_{\Vh}^2 \lesssim \sum_{K\in\Th} \big\Vert P_K u \Vert_{\Vh}^2 \leq \frac{1}{c_{\text{\tiny\eqref{eq:quasiortho:c}}}^2} \sum_{K\in\Th} \big\Vert \AK  P_K u \Vert_{\QKh'}^2
        = \frac{1}{c_{\text{\tiny\eqref{eq:quasiortho:c}}}^2} \sum_{K\in\Th} \big\Vert \AK  u \Vert_{\QKh'}^2.
    \end{align*}
\end{proof}
In some settings, such as the non-conforming setting considered in \cref{ssec:DGsetting} the projections $P_K$ are naturally restriction operators of functions and 
conditions \eqref{eq:quasiortho:a} and \eqref{eq:quasiortho:b} are easily fulfilled and it remains only to \emph{locally} check \eqref{eq:quasiortho:c}. 
In other situations it may be more difficult to construct the projections $P_K$. In the next subsection we give a generic construction of $P_K$ based on an additional continuity assumption.

\subsection{Local stability properties inherited from prototype operators}\label{sec:protop}

\cref{ass:local}, respectively \eqref{eq:quasiortho:c}, may be challenging to verify for a general operator $\AK $. 
To verify the assumptions needed for the theory it is often beneficial to choose a well-understood operator $A_{K,0}$ that approximates $\AK$.
The following lemma shows that if \eqref{eq:local_solveability} holds for a prototype operators $A_{K,0}$, then it also holds for $\AK $ if the distance between them is small enough.

\begin{lemma}\label{lem:abstractneumann}
    If for some invertible  prototype operator $A_{K,0}:\UKh \to \QKh'$ there exist constants $\omega \ne 0$ and $\gamma \in (0,1)$ such that
    \begin{equation}\label{eq:nearlycontant}
        \norm{\omega \AK  A_{K,0}^{-1} - \id}_{\QKh' \to \QKh'}    = \sup_{u\in \UKh} \frac{\norm{\omega \AK  u - A_{K,0}u}_{\QKh'}}{\norm{A_{K,0}u}_{\QKh'}}\le\gamma,
    \addtocounter{equation}{1} \tag{$A_{K,0}|\theequation$}
    \end{equation}
    then the restriction $\AK : \UKh \to \QKh'$ is invertible, with
    \[
        c_{\text{\tiny\eqref{eq:quasiortho:c}}}^{-1} = \sup_{u\in \UKh} \frac{\norm{u}_{\Vh}}{\norm{\AK  u}_{\QKh'}} \le \frac1\omega \frac{1}{1-\gamma} \sup_{u\in \UKh} \frac{\norm{u}_{\Vh}}{\norm{A_{K,0} u}_{\QKh'}} .
    \]
\end{lemma}
\begin{proof}
    Without loss of generality, by replacing $\AK $ by $\omega \AK $, we assume that $\omega=1$. We define $Y = \id - \AK  A_{K,0}^{-1}$ and in the remainder of the proof we will use the short hand notation $\Tnorm{\cdot}$ for the operator norm 
    $\norm{\cdot}_{\QKh' \to \QKh'}$ for operators mapping from $\QKh'$ to $\QKh'$. Then 
    \[
    \begin{aligned}
        \Tnorm{Y} & = \Tnorm{\AK  A_{K,0}^{-1} - \id}     
        = \sup_{q\in \QKh'}{\norm{(\AK  - A_{K,0})A_{K,0}^{-1} q }_{\QKh'}} \Big/{\norm{q}_{\QKh'}} \\
        & \!\!\!\!\!\!\stackrel{u = A_{K,0}^{-1}q}{=} \sup_{u\in \UKh}{\norm{(\AK  - A_{K,0})u }_{\QKh'}} \Big/ {\norm{A_{K,0}u}_{\QKh'}} \le \gamma <1.
    \end{aligned}
    \]
    Hence, the corresponding Neumann series converges, i.e.
    \begin{align*}
        (\id - Y)^{-1} = \sum_{k=0}^\infty Y^k \quad \text{and}\quad  \sum_{k=0}^\infty \Tnorm{Y^k} \le \sum_{k=0}^\infty \gamma^k = \frac{1}{1-\gamma},
    \end{align*}
    and it follows that $\AK  A_{K,0}^{-1} = (\id - Y)^{-1}$ is invertible. Further we get with 
    $\Tnorm{(\id - Y)^{-1}} \le \sum_{k=0}^\infty \Tnorm{Y^k}$ the bound  $\Tnorm{( \AK A_{K,0}^{-1})^{-1}} \le \frac{1}{1-\gamma}$.
    Finally, we conclude that 
    \[\AK ^{-1}= A_{K,0}^{-1} (\AK  A_{K,0}^{-1} )^{-1}, \] 
    exists and satisfies 
    \begin{align*}
        c_{\text{\tiny\eqref{eq:quasiortho:c}}}^{-1} & = \sup_{u\in \UKh} \frac{\norm{u}_{\Vh}}{\norm{\AK  u}_{\QKh'}} = \norm{\AK ^{-1}}_{\QKh' \to \UKh}
         \le \frac{1}{1-\gamma} \norm{A_{K,0}^{-1}}_{\QKh' \to \UKh},
    \end{align*} 
    which implies the claim.
\end{proof}

\subsection{The discontinuous Galerkin setting} \label{ssec:DGsetting}

Let $\Omega \subset \mathbb{R}^{d}, d=2,3$ be a bounded Lipschitz domain and $\Th$ be a partition of $\Omega$ into non-overlapping elements $K$.
We assume that we are interested in approximating the solution $u$ of a partial differential equation of the form $ A u = f $ in $\Omega$ with suitable boundary conditions and further assume that function spaces $V$ and $W$ over $\Omega$ (e.g. Sobolev spaces) are given such that the PDE problem is well-posed in a weak form: Find $u \in V$ s.t. $ a (u, \cdot) = \ell(\cdot)$ in $W'$.  

In the discontinuous Galerkin setting, we consider discrete spaces that are allowed to be non-conforming, i.e. $\Vh\not\subset V$ and $\Wh\not\subset W$. Further, in the following we only consider the case $\Wh = \Vh$. The space $\Vh$ is constructed from local spaces $\Vh(K)$ on each element $K\in\Th$, then
\[
\Vh = \bigoplus_{K\in \Th} \Vh(K).
\]
We assume that a corresponding localization also holds for the decomposition $\Vh =\ITh \oplus \Uh$, which is accordingly translated to the local space, i.e. we have  $\Vh(K) = \ITh(K) \oplus \UKh$ with $\ITh = \bigoplus_{K\in \Th} \ITh(K)$ (and $\Uh = \bigoplus_{K\in \Th} \UKh$).
For $P_K$ we choose the restriction operator\footnote{The restriction $u|_{K}\in\Vh$ is extended by zero outside of $K$.} $P_K u = \restr{u}{K}$ for all $u\in \Vh$ and $K\in\Th$ which guarantees \eqref{eq:quasiortho:b}.
Further, with the next assumption we assume that the operators $\AK $, often the (scaled) restriction of the strong form operator $A$ on $K$, effectively act only on $\Vh(K)$.
\begin{assumption}[Locality of $\AK $] \label{ass:locality}
    For all $v_h \in \Vh$  
    we have the strong locality property of $\AK $
    \begin{equation}\label{eq:simple_ass_strloc}
        \AK  ~ (v_h|_{K'}) = 0 \text{ in } \QKh' \quad \forall K' \neq K,
    \end{equation}
    for a suitable local space $\QKh$.
\end{assumption}
\cref{ass:locality}   together with $P_K \cdot = \restr{\cdot}{K}$ ensures \eqref{eq:quasiortho:a}. 
We note that $\QKh$ should be chosen such that \cref{ass:locality} holds. 
This choice is still open and depends on the choice of $\AK $.

From an algorithmic point of view it is worth noting that \eqref{eq:global-onT} decomposes into a set of local problem: 
Find $u_\IL \in \Uh$ with $u_\IL = \sum_{K \in \Th} u_{\IL,K},~ u_{\IL,K} \in \UKh$ such that
\begin{equation*}
    \langle \AK  u_{\IL,K}, q_K\rangle = \tilde\ell_K(q_K) \quad \forall K\in \Th, q_K\in \QKh.
\end{equation*}

This assumption allows us to simplify the requirements that we need to verify in \cref{lemma:quasiortho} and therefore we get the following simple corollary.
\begin{corollary}\label{cor:locality}
   Assume that \eqref{eq:quasiortho:c} and \cref{ass:locality} hold, then \cref{ass:local} also holds true.
\end{corollary}
\begin{proof}
To apply \cref{lemma:quasiortho} we use that $P_K u_h= \restr{u_h}{K}$, for all $K\in\Th$. 
For this choice \eqref{eq:quasiortho:b} follows from the Cauchy-Schwarz inequality.
Additionally, equation \eqref{eq:simple_ass_strloc} implies \eqref{eq:quasiortho:a}. 
Hence, since we considered that \eqref{eq:quasiortho:c}, \cref{lemma:quasiortho} can be applied, and thus we obtain \cref{ass:local}. 

\end{proof}

\subsection{Generic construction for the embedded Trefftz method}\label{ssec:genericembt}

Typically, when constructing the embedded Trefftz space \eqref{eq:IThproto} we have a suitable test space $\QKh$ in mind.
In this subsection we present a generic approach to obtain candidates that provide \cref{ass:local} in an embedded Trefftz DG method. 

Let $\QK$ be given with a proper inner product, often this will be $L^2(K)$, and 
$q_1,\ldots,q_m$ be an orthonormal basis of a sufficiently rich subspace of $\QK$. 
We aim to distinguish functions in $\Vh(K)$ between \emph{kernel}-like functions and remainder with respect to the operator $\AK $ based on suitable SVD decomposition.
Let $v_1,\ldots,v_n$ be an orthonormal basis of $\Vh(K)$.
Consider an SVD decomposition of
\[
(\langle \AK  v_i, q_j\rangle)_{ij} = Q_v \Sigma Q_q^T.
\]
This gives us a new orthonormal basis $\tilde v_1,\ldots,\tilde v_n$ of $\UKh$ and orthonormal vectors $\tilde q_1,\ldots, \tilde q_m$ corresponding to descending singular values.
We obtain the structure
\[
(\langle \AK  \tilde v_i, \tilde q_j\rangle)_{ij} = \Sigma 
 = \diag(\sigma_1,...,\sigma_{\min(n,m)},0,...,0) \in \mathbb{R}^{n\times m},
\]
with $\sigma_1\geq \sigma_2 \geq \ldots \geq \sigma_{\min(n,m)} \geq 0$.

We fix a threshold $\tau$ and choose $k$ such that $\sigma_{k+1} < \tau \leq \sigma_{k}$
with $\sigma_k$ being the smallest singular value representing the local problem. 
As a basis of $\UKh$, we choose the first $k$ vectors $\tilde v_1,\ldots,\tilde v_k$ and as a basis of $\QKh$, we choose the first $k$ vectors $\tilde q_1,\ldots,\tilde q_k$. As the complement of $\UKh$ we take $\ITh(K)  = \langle \tilde v_{k+1},\ldots,\tilde v_n\rangle$. Since by construction we always have $\langle \AK  \tilde v_{i}, \tilde q_j \rangle = 0$ for $i \neq j $, we get for the space $\ITh(K)$ the desired property
\[
\langle \AK  u, q \rangle = 0 \quad \forall u\in \ITh(K), q\in \QKh.
\]
\begin{lemma}\label{lem:weirdconstruction:nice}
    Choosing the threshold $\tau = c_{\text{\tiny\eqref{eq:quasiortho:c}}}$, the spaces $\ITh(K)$, $\UKh$ and $\QKh$ constructed above together with $A_K$ yield 
    \begin{equation*}
        c_{\text{\tiny\eqref{eq:quasiortho:c}}} \norm{u}_{\Vh}  \le \norm{\AK  u}_{\QKh'} ,
        \end{equation*}      
               which implies \cref{ass:local}.
\end{lemma}
\begin{proof}
    By construction we have that for the singular values of $(\langle \AK  \tilde v_i, \tilde q_j\rangle)_{i,j=1,..,k}$ there holds
    \[
        \sigma_1 \geq \ldots \geq \sigma_k \geq c_{\text{\tiny\eqref{eq:quasiortho:c}}},
    \]
    which implies the stability bound of \eqref{eq:quasiortho:c}.

    Now choose $P_K$ with $P_K u = P_K u|_K$ as the orthogonal projection to $\UKh$. By construction of $\UKh$ and $\QKh$, we have that $\langle \AK  P_K u, q \rangle = \langle \AK  u, q \rangle$ for all $u\in \Vh$, $q\in \QKh$. Hence we can apply \cref{lemma:quasiortho}.

\end{proof}

\section{Applications of the framework} \label{sec::examples}
In this section we present some examples of Trefftz DG methods that can be analyzed within the framework of this paper.
We recall from the discussion in \cref{ssec:DGsetting} that the main assumptions that need to be verified are \cref{ass:locality} and equation \eqref{eq:quasiortho:c} to obtain \cref{ass:local}.
Several examples are accompanied by numerical studies.
For the implementation of the methods we are using \texttt{NGSolve} \cite{ngsolve} and \texttt{NGSTrefftz} \cite{ngstrefftz}\footnote{Reproduction material is available in \cite{llsv_replication}.}.

We start by introducing some notation and assumptions on the index set $\Th$.
Let $\Th = \{K\}$ be a division of $\Omega$ into non-overlapping elements $K$. 
The local mesh size of a mesh element $K\in\Th$ is defined as $h_K = \diam(K) \coloneqq \sup_{x_1, x_2 \in K} |x_1 - x_2|$.
In a slight abuse of notation we write $h = \sup_{K \in \Th} h_K$.
We assume to work with mesh sequences that satisfy the following properties:
\begin{enumerate}[label={M\arabic*.},ref=M\arabic*]
    \item\label{ass:mesh1} There are two balls $b_{K'} \subset {K'} \subset B_{K'}$, such that ${K'}$ is star-shaped with respect to the ball $b_{K'}$ and $\diam(B_{K'})/\diam(b_{K'}) \lesssim 1$.
    \item\label{ass:mesh2} The element boundary can be divided into mutually exclusive subsets $\{ F_i \}_{i=0}^{n_{K'}}$ with $\diam(K') \leq c \diam(F_i)$, $i=0,...,n_{K'}$, where $n_{K'}$ and $c$ are uniformly bounded, satisfying
  \begin{itemize}
    \item[(i)] There exists a sub-element ${K'}_{F_i}$ of $K'$ with $d$ planar facets meeting at a
    vertex $x^0_i\in K'$, such that  ${K'}_{F_i}$ is star-shaped with respect to $x^0_i$ and $h_{K_{F_i}'} \simeq h_{K'}$.
    \item[(ii)] There exists a uniform constant $c_{\eqref{eq:gmesh}}$, such that 
        \begin{equation}\label{eq:gmesh}
      (x - x^0_i) \cdot n_{F_i}(x) \geq c_{\eqref{eq:gmesh}} h_{K'} \quad \forall x \in F_i.
    \end{equation}
  \end{itemize}
\end{enumerate}

Next, we introduce some standard notation for the DG method.
We denote by $\calF_h$ the set of all facets of $\calT_h$, i.e. the union of all $(d-1)$ dimensional parts of the element boundaries.
We assume that each $F\in\calF_h $ is either an \textit{interior facet} for which there exist two distinct elements $K_1, K_2\in \calT_h$ such that $F=\partial K_1\cap \partial K_2$, or a \textit{boundary facet} for which there exists an element $K_1\in\calT_h$ such that $F\subset \partial K\cap \partial \Omega$.
The sets of interior facets are denoted by $\Fhi$.
On each face $F\in\Fh$, we define the normal vector $n_F$ as the outward unit normal vector to $K_1$. 
On interior facets jump and average operators are defined as $\jmp{u} = u|_{K_1} - u|_{K_2}$ and $\avg{u} = \frac{1}{2}(u|_{K_1} + u|_{K_2})$, respectively.
On boundary facets we set $\jmp{u} = \avg{u} = u|_{K_1}$.

In this section we consider the underlying DG methods to be defined on the element-wise polynomial space, thus we set
\begin{equation*}
    \Vh = \IP^p(\Th) = \{v\in L^2(\Omega) \mid v|_K \in \IP^p(K) \quad \forall K\in\Th\},
\end{equation*}

For the reader's convenience, we identify $L^2$ with its dual, however, this identification cannot be done on subspaces.

\begin{remark}[Implementation of the embedded Trefftz method for known \texorpdfstring{$\QKh$}{QKh}]
Assume that a suitable choice for $\QKh$ is given, i.e. it is chosen such that for $\ITh(K)$ as in \eqref{eq:IThproto} there exists a space $\UKh$ such that \cref{ass:local,ass:local_strong} are fulfilled.

The Trefftz space $\ITh(K)$, given by \eqref{eq:IThproto}, is constructed as the kernel of 
the matrix $\bA=(\langle \AK  v_j, q_i \rangle)_{ij}$ for a basis $v_1,\dots,v_n$ of $\VKh$ and a basis $q_1,\dots , q_k$ of $\QKh$.
Note that, under the assumptions, the matrix has full row rank.
Thus, the dimension of the Trefftz space is given by $\dim \ITh(K) = \dim \Vh(K) - \dim \QKh$.

The kernel and a pseudo-inverse of $\bA$ can be computed using singular value decomposition.
The pseudo-inverse can be used to solve the local problem \eqref{eq:global-onT} and to obtain a particular solution.
It guarantees that the solution will be in a complementary space to the kernel $\ITh(K)$.
Note that, due to \cref{rem:indepL}, it is irrelevant whether the image of the pseudo-inverse is in the exact space $\UKh$ used for the analysis or a different complementary space. 
\end{remark}

\subsection{Embedded Trefftz DG for the advection-reaction equation}\label{sec:ar}
We briefly recall the DG discretization of the advection-reaction equation, following along the lines of \cite[Section 2]{DiPietroErn}.
Let us consider the advection-reaction equation 
\begin{eqs}\label{eq:arpde}
    \beta\cdot \nabla u + \gamma u &=f &&\quad \text{in } \Omega,\\
    u &= g_D &&\quad \text{on } \partial \Omega^-,
\end{eqs}
where $\partial \Omega^-:=\{x\in\partial\Omega\mid\beta\cdot\nf<0\}$ is the inflow boundary of $\Omega$
and the inflow boundary data $g_D\in H^{\frac12}(\partial\Omega^-)$.
Furthermore, let $\beta$ be Lipschitz in each component, i.e. $\beta\in [\mathrm{Lip}(\Omega)]^d$, and assume that there exists a constant $\gamma_0>0$ such that 
\begin{equation}\label{assumptiondiv}
    \gamma(x)-\frac12\mathrm{div}\big(\beta(x)\big)\ge \gamma_0  \quad \text{ a.e. } x \in\Omega.
\end{equation}
For simplicity we also assume $c_\gamma \leq \gamma(x) \leq C_\gamma$ and $c_\beta \leq \|\beta(x)\|_2 \leq C_\beta$ for $x\in\Omega$ for some constants $c_\gamma, C_\gamma, c_\beta, C_\beta>0$ and do not track the dependence of the constants on $c_\beta$, $C_\beta$, $c_\gamma$, $C_\gamma$ in the following.
Here, $\|\cdot\|_2$ denotes the Euclidean norm in $\mathbb{R}^d$.
Considering a source term $f\in L^2(\Omega)$, there exists a unique solution $u\in L^2(\Omega)$ with $\beta\cdot\nabla u\in L^2(\Omega)$.

The upwinding DG discretization of the advection-reaction equation then reads as follows: 
\begin{equation}\label{eq:arvar}
\text{Find } u_h \in \Vh \text{ such that }
a_h(u_h,v_h)=\ell_h(v_h) \quad \forall v_h \in \Vh,
\end{equation}
with the DG bilinear form for the advection-reaction equation
\begin{eqs} \label{eq:arah}
    a_h(u_h,v_h) := 
    &\inner{\beta \cdot \nabla u_h, v_h}_{\Th} + \inner{\gamma u_h, v_h}_{\Th}
    - \inner{(\beta\cdot\nf)\jmp{u_h},\avg{v_h}}_{\Fhi} \\
    &
    -\inner{(\beta\cdot\nf) u_h,v_h}_{\partial\Omega^-}
    + \frac12 \inner{|\beta\cdot \nf| \jmp{u_h},\jmp{v_h}}_{\Fhi},
\end{eqs}
and the linear form 
\begin{equation}\label{eq:arlh}
    \ell_h(v_h):=
    \inner{f,v_h}_\Th - \inner{g_D \beta\cdot\nf,v_h}_{\partial\Omega^-}.
\end{equation}

We now present the embedded Trefftz DG method for this setting. 
Based on the discussion in \cref{sec:embtrefftz} we choose for $v_h \in \Vh$ the operator $\AK $ as the localized strong form operator of the advection-reaction equation, i.e.
\begin{equation}\label{eq:arak}
\AK  v_h = h_K^{1/2} (\beta \cdot \nabla v_h + \gamma v_h) \quad \text{and} \quad 
\ell_K = h_K^{1/2} f.
\end{equation}
We set $\QKh = \IP^{p-1}(K)$ as the local test space, a choice motivated by a suitable prototype operator, which will be further elaborated in \cref{rem:ar_qkh}.
We note that $\QKh$ and the scaling in $\AK $ are chosen so that \cref{ass:local_strong} holds.
Now, as in \cref{sec:embtrefftz}, the local Trefftz space is given by \eqref{eq:IThproto}, resulting in
\begin{align*}
    \ITh(K)&= \{v_h \in \IP^p(K) \mid \inner{\AK  v_h, q_h}_K=0,\ \forall q_h\in \IP^{p-1}(K) \} .
\end{align*}
As discussed in \cref{sec:embtrefftz} this guarantees \cref{ass:inexacttrefftz} with $\rho = 0$. 

The embedded Trefftz DG method for the advection-reaction problem then reads:
\begin{equation}\label{eq:Tarvar}
\begin{alignedat}{3}
    \text{Find } u_h \in \Vh 
    \text{ such that } &&\quad\inner{\AK  u_h, q_h}_K&=\inner{h_K^{1/2}\!f,q_h}_K &&\quad \forall q_h\in \IP^{p-1}(K), K\in\Th,
 \\ \text{ and }  && \quad a_h(u_h,v_h)&=\ell_h(v_h) &&\quad \forall v_h \in \ITh,
\end{alignedat}
\end{equation}
with $a_h(\cdot,\cdot)$ given in \eqref{eq:arah} and $\ell_h(\cdot)$ given in \eqref{eq:arlh}.

For the analysis we define the norms
\begin{eqs}
\norm{v}^2_{\Vh}:=& \norm{v}^2_{\Omega}+ \sum_{F\in\calF_h} \norm{|\beta_r\cdot\nf|^{\frac12} \jmp{v}}_F^2 + 
\sum_{K\in\Th} h_K \norm{\beta_r\cdot\nabla v}^2_{K},
\\
\norm{v}^2_{\Vsh}:=&\norm{v}^2_{\Vh}+\sum_{K\in\Th} (h_K^{-1}\norm{v}^2_{K} + \norm{v}^2_{\partial K}),
\end{eqs}
where  $\beta_r = \beta / \norm{\beta}_{L^\infty(\Omega)}$. From the definition of the norms \cref{ass:locality} is obvious. Together with the DG setting this implies \eqref{eq:quasiortho:a} and \eqref{eq:quasiortho:b}. To obtain \eqref{eq:quasiortho:c}, and hence \cref{ass:local}, we make use of a prototype operator.

\subsubsection{The prototype operator for \texorpdfstring{\cref{lem:abstractneumann}}{}}\label{sec:arproto}

We introduce a prototype operator in order to make use of \cref{lem:abstractneumann}. 
This will allow us to show that \eqref{eq:quasiortho:c} holds for $\AK $.
We define the prototype operator for the advection-reaction problem as the version of $\AK $ that corresponds to locally constant coefficients in the PDE and neglects the lowest order term, i.e.
\begin{eqs}\label{eq:acak0}
    A_{K,0} v_h = h_K^{1/2} \bar \beta \cdot\nabla v_h ,
\end{eqs}
where $\bar \beta=\beta(x_K)$, where $x_K$ in a given point in the element $K\in\Th$.

\begin{remark}[Choice of $\QKh$ and optimality of $\IT_h$]\label{rem:ar_qkh}
    We observe that 
    \[
        \{A_{K,0} v\mid \forall v\in\Vh\}= \QKh=\IP^{p-1}(K),
    \] 
    which finally explains the choice of $\QKh$ in the definition of $\ITh$.
    Further, we note that the minimal dimension of the Trefftz space for the operator $A_{K,0}$ is $\dim \Vh(K)-\dim\QKh$.
    The dimension of $\ITh$ is optimal in that sense.
\end{remark}

\begin{figure}[!ht]
\begin{center}
\begin{tikzpicture}[>=latex, x=1.5cm, y=1.5cm, scale=.9, font=\footnotesize]
    \def\length{10*sqrt(1+(x+y)^2)}
    \begin{axis}[
        axis line style={draw=none},
        tick style={draw=none},
        xticklabel=\empty,
        yticklabel=\empty,
        view={0}{90},
        domain=0:1,
        samples=10,
        axis equal image,
    ]
    \addplot3[
        gray,
        quiver={
            u={1/(\length)},
            v={(x+y)/(\length)},
        scale arrows=1.0,
        every arrow/.append style={-latex}},
    ] {x};
    \addplot3[
        black,
        quiver={
            u={0.012/0.184},
            v={0.013/0.184},
        scale arrows=1.0,
        every arrow/.append style={-latex}},
    ] {x};
    \draw (0.0,0.2) node[anchor=north]{}
       -- (1.0,0.0) node[anchor=north]{}
       -- (0.6,1.0) node[anchor=south]{}
       -- cycle;
   \draw[gray, fill=teal, opacity=0.2] (0.5204,0.3988) circle[radius=0.297];
   \node[fill,circle,red,inner sep=1.5pt] (a) at (0.5204,0.3988) {};
   \node[] (b) at (-0.013/0.184 + 0.5204, 0.012/0.184 + 0.3988) {};
   \draw[red] (a) -- ($(a)!3cm!(b)$) -- ($(b)!4cm!(a)$);
    \end{axis}

    \draw[red] (4,3.5) -- (4.2,3.5);
    \node[anchor=west] at (4.2,3.5) {$\Gamma_{K,\beta}$};
    \draw[arrows=-stealth,gray] (4,3) -- (4.2,3);
    \node[anchor=west] at (4.2,3) {$\beta$};
    \draw[arrows=-stealth,black] (4,2.5) -- (4.2,2.5);
    \node[anchor=west] at (4.2,2.5) {$\bar\beta$};
    \draw[gray, fill=teal, opacity=0.2] (4.1,2.0) circle[radius=0.1];
    \node[anchor=west] at (4.2,2) {$B(K)$};
    \node[fill,circle,red,inner sep=1.5pt] (a) at (4.1,1.5) {};
    \node[anchor=west] at (4.2,1.5) {$x_K$};

\end{tikzpicture}
\begin{tikzpicture}[>=latex, x=1.5cm, y=1.5cm, scale=.9, font=\footnotesize]
    \def\length{10*sqrt(1+(x+y)^2)}
    \begin{axis}[
        axis line style={draw=none},
        tick style={draw=none},
        xticklabel=\empty,
        yticklabel=\empty,
        view={0}{90},
        domain=0:1.2,
        samples=9,
        axis equal image,
    ]
    \addplot3[
        black,
        quiver={
            u={1},
            v={0},
        scale arrows=0.08,
        every arrow/.append style={-latex}},
    ] {x};
   \draw[gray, fill=teal, opacity=0.2] (0.6,0.6) circle[radius=0.5];
   \draw[gray, thick, ->] (0.6,0) -- (0.6,1.2) node[anchor=north west] {$x_2$};
   \draw[gray, thick, ->] (0,0.6) -- (1.2,0.6) node[anchor=south east] {$x_1$};
   \draw[red, thick] (0.6,1.2) -- (0.6,0);
   \node[fill,circle,red,inner sep=2.5pt] (a) at (0.6,0.6) {};
   \node[anchor=north west] at (0.6,0.6) {$0$};

   \draw[gray]
   (0.6,0.6) -- (0.95355339,0.95355339);   
   \draw[decoration={brace,raise=0pt},decorate]
   (0.6,0.6) -- node[above left =0.05 pt] {$1$} (0.95355339,0.95355339);   
    \end{axis}

    \draw[red,thick] (4.2,3.5) -- (4.4,3.5);
    \node[anchor=west] at (4.4,3.5) {$\hat \Gamma$};
    \draw[arrows=-stealth,black] (4.2,3) -- (4.4,3);
    \node[anchor=west] at (4.4,3) {$\hat\beta = (1,0,..)$};
    \draw[gray, fill=teal, opacity=0.2] (4.3,2.5) circle[radius=0.1];
    \node[anchor=west] at (4.4,2.5) {$B_1$};

\end{tikzpicture}
\end{center}
\caption{
    Illustration of the flow fields $\beta$ and its approximation $\bar\beta$ in a triangle and the in-circle $B(K)$ as well as the hyperplane $\Gamma_{K,\beta}$ orthogonal to $\bar \beta$ (left). After translation, rotation and rescaling the configuration on an arbitrary element $K$ is mapped to a reference configuration (right).}
\label{fig:prototype}
\end{figure}
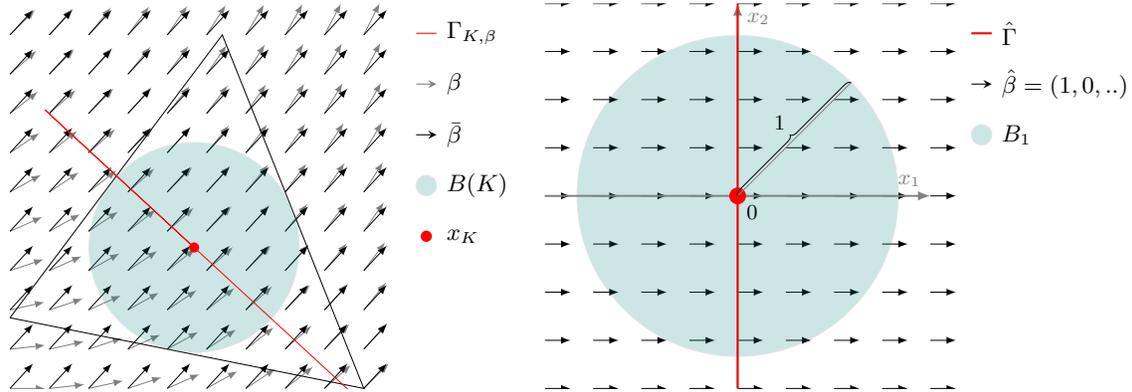
We now make the (arbitrary) choice of $x_K$ to be the center of the in-(hyper)circle $B(K)$ of the element $K$. 
In order to define a suitable space $\UKh$ we introduce on each element a hyperplane $\Gamma_{K,\beta}$, orthogonal to $\bar\beta$ passing through $x_K$.
The following lemma displays the crucial structure that we are going to exploit when 
defining suitable spaces $\UKh$, $\QKh$, i.e. that the $\Vh$-norm can be controlled by the sum of two parts: the $L^2$-norm of the directional derivative in the direction of $\bar \beta$ and the $L^2$-norm of the function on the hyperplane $\Gamma_{K,\beta}$.
\begin{lemma} \label{lem:loclemma}
    There holds 
\begin{align}
    \norm{u_h}_{\Vh}^2 \lesssim h_K^{-1} \norm{u_h}_{K}^2 \lesssim h_K \norm{\bar \beta \cdot \nabla u_h}_{K}^2 + \norm{u_h}_{\Gamma_{K,\beta}\cap B(K)}^2 \quad \forall u_h \in \IP^p(K), K \in \Th,
\end{align}
with $u_h$ extended by zero outside of $K$.
\end{lemma}
\begin{proof}
    For $K\in\Th$ we denote by 
    $B(K)$ the in-(hyper)circle with center $x_K$ and radius $\rho_K$.
    Let further $Q \in \mathbb{R}^{d\times d}$ be a scaled (by $\norm{\bar\beta}_2$) orthogonal matrix that rotates 
    $\bar \beta$ to the unit vector $(1,0,\ldots)\in\mathbb{R}^d$. 
    With the invertible affine map $\Psi_K: B(K) \to B_1$, $\Psi_K(x) = \rho_K^{-1} Q \cdot (x-x_K)$ where $B_1$ is the unit ball around the origin, cf. \cref{fig:prototype} for a sketch, we define $\hat u_h = u_h \circ \Psi_K^{-1}$ for $u_h \in \IP^p(K)$ and estimate
    \begin{align*}
        \norm{u_h}_{\Vh}^2 & = \norm{u_h}_{K}^2 + \big\Vert |\beta_r \cdot n|^{1/2} u_h \big\Vert_{\partial K}^2 + h_K\norm{\beta_r\cdot\nabla u_h}_{K}^2 \lesssim 
        h_K^{-1} \norm{u_h}_{K}^2
        \lesssim h_K^{-1} \norm{u_h}_{B(K)}^2 = \dots
        \intertext{
            Here, we made use of standard inverse inequalities to go from the element $K$ to its in-(hyper)circle $B(K)$. Applying a transformation rule for $\Psi_K$ with $\operatorname{det}(\nabla\Psi_K) = \rho_K^{-d}$ we obtain
        }
        \dots & = h_K^{-1} \rho_K^{d} \norm{\hat u_h}_{B_1}^2
        \simeq h_K^{-1} \rho_K^{d} \big( \norm{ \partial_{x_1} \hat u_h}_{B_1}^2 + \norm{\hat u_h}_{\hat \Gamma}^2 \big) \lesssim \dots 
        \intertext{
            with $\hat \Gamma = \{ x \in B_1 \mid x_1 = 0 \} = \Psi_K (\Gamma_{K,\beta})$
            where we exploited norm equivalence of all norms on the full polynomial space with a constant independent of $K$. Transforming back we obtain
        }
        \dots  & \lesssim h_K \norm{ {\bar \beta} \cdot \nabla u_h}_{B(K)}^2 + \norm{u_h}_{\Gamma_{K,\beta}\cap B(K)}^2 
        \lesssim h_K \norm{ {\bar \beta} \cdot \nabla u_h}_{K}^2 + \norm{u_h}_{\Gamma_{K,\beta}\cap B(K)}^2,
    \end{align*}
    where we again made use of inverse inequalities.
\end{proof}

We can conclude that only the trivial function vanishes both on $\Gamma_{K,\beta}$ and under application of $A_{K,0}$, respectively the directional derivative $\bar \beta \cdot \nabla$.
Hence, we define $\UKh$ as
\begin{align*}
    \UKh=\{ v_h\in\IP^{p}(K) \mid \restr{v_h}{\Gamma_{K,\beta}}=0\}.
\end{align*}
Note again that the test space which matches the range of $\UKh$ is exactly
$\QKh=A_{K,0}\UKh=\IP^{p-1}(K)$.
To apply \cref{lem:abstractneumann} we now show that $A_{K,0}$ is bijective.
\begin{lemma}
    The operator $A_{K,0}:\UKh\rightarrow\QKh$ is invertible and has a bounded inverse.
Furthermore, for $u_h \in \UKh$ we have 
$$\norm{A_{K,0} u_h}_{\QThp} = \|h_K^{\frac{1}{2}} \bar \beta \cdot \nabla u_h\|_{\QThp} \gtrsim \norm{u_h}_{\Vh}^2.$$
\end{lemma}
\begin{proof}
    We first prove invertibility. Take $w_h\in \QKh=\IP^{p-1}(K)$.
    To $x \in K$, let $(\xi,\eta) \in \mathbb{R} \times \Gamma_{K,\beta}$ be the coordinates so that $\xi$ is the (signed) distance (oriented with $\bar\beta$) to the hyperplane $\Gamma_{K,\beta}$ and $x=\eta + \xi n_{\Gamma_{K,\beta}}$ so that $\bar \beta \cdot \nabla v_h(x) = |\bar \beta| \partial_\xi v_h(\eta +\xi n_{\Gamma_{K,\beta}})$ for any $v_h \in \IP^p(K)$.
    We can then define 
    $
    u_h(x) = u_h(\eta+\xi n_{\Gamma_{K,\beta}}) = \frac{h_K^{-1/2}}{|\bar\beta|} \int_{0}^{\xi} w_h(\eta+s n_{\Gamma_{K,\beta}}) \, ds \in \UKh
    $
    such that $w_h=A_{K,0} u_h = h_K^{1/2} \bar\beta\cdot\nabla u_h$. Further, we have for any $u_h \in \UKh$ with $w_h=A_{K,0} u_h$
\begin{align*}
    \Vert w_h \Vert_K^2 = \Vert A_{K,0} u_h \Vert_K^2 = \inner{A_{K,0} u_h,w_h}_K 
    &= \big\Vert h_K^{1/2}\bar\beta\cdot\nabla u_h \big\Vert_{K}^2 \gtrsim \norm{u_h}_{\Vh}^2,
\end{align*}
where we made use of \cref{lem:loclemma} in the last step.
\end{proof}

Since $A_{K,0}:\UKh\rightarrow\QKh$ is bijective we have that $\dim(\Uh) + \dim(\ITh) = \dim(\Vh)$ and therefore $\Vh = \Uh \oplus \ITh$ is a valid decomposition.

Next, we will show that for the choice $A_{K,0}$ for the prototype operator and the associated spaces the requirements in \cref{lem:abstractneumann} are satisfied.

\begin{lemma} \label{lem::ar_m_ar0} 
    We have for all $u_{\IL}\in \UKh$ that with $\omega = 1$ there holds
    \begin{align*}
        \| \omega \AK  u_{\IL} - A_{K,0} u_{\IL} \|_{K} \lesssim h_K \norm{u_{\IL}}_\mathrm{\Vh}.
    \end{align*}
\end{lemma}
\begin{proof}
    With the Lipschitz bound on $\beta$ we have for all $u_{\IL} \in \UKh$
    \[ 
    \begin{aligned}
        h_K^{-1/2}\| \AK   u_{\IL} -  A_{K,0}u_{\IL} \|_{K}   
                                       &  \leq \norm{\beta - \bar\beta}_{L^\infty(K)} \norm{u_{\IL}}_{H^1(K)} + C_\gamma \norm{u_{\IL}}_{K} 
                                          \lesssim (L_\beta + C_\gamma) \norm{u_{\IL}}_{K} \\
                                       & \lesssim \norm{u_{\IL}}_{B(K)}
                                         \lesssim h_K \norm{ {\bar \beta} \cdot \nabla u_{\IL}}_{B(K)}
                                         \lesssim h_K^{1/2} \norm{ A_{K,0} u_{\IL}}_{K},
    \end{aligned}
    \]
    where $L_\beta$ is the Lipschitz constant of $\beta$ and $C_\gamma$ is the $L^\infty$-norm of $\gamma$ on $K$.
    For the last line of estimates we used a scaling argument, made use of \cref{lem:loclemma} and $u_{\IL}=0$ on $\Gamma_{K,\beta}$, and finally the definition of the $\Vh$-norm.
\end{proof}

Using \Cref{lem:abstractneumann} together with the last two lemmas we conclude that  \eqref{eq:quasiortho:c} holds for $A_{K}$ and hence \cref{ass:local} holds.

\subsubsection{Coercivity on the Trefftz space}

It only remains to check \cref{ass:assumption_Th} before we can conclude stability of the coupled problem.
We start with a preparatory lemma:
\begin{lemma}\label{lem:arnormeq}
    For all $u\in\ITh$ we have that 
    \begin{eqs}
        \sum_{K\in\Th}\norm{\beta \cdot\nabla u}_{K}^2 \leq C_{\beta,\gamma}\sum_{K\in\Th} \norm{u}_{K}^2,
    \end{eqs}
    with $C_{\beta,\gamma}=\abs{\beta}_{\mathrm{Lip}(\Omega)} + \norm{\gamma}_{L^\infty(\Omega)}$ and independent of $h_K$.
\end{lemma}
\begin{proof}
    For $u\in\ITh$ we have that $\AK = \Pi^{p-1}\AK  u=\Pi^{p-1}\beta\cdot\nabla u + \Pi^{p-1}\gamma u = 0$ and therefore,
    for any $K\in\Th$
\begin{align*}
    \norm{\beta \cdot\nabla u}_{K} 
    &\leq \norm{(\beta - \bar\beta) \cdot\nabla u}_{K} + \norm{ \bar\beta\cdot\nabla u - \Pi^{p-1} \beta\cdot\nabla u}_{K} + \norm{\Pi^{p-1}\gamma u}_{K}\\
    &\leq \norm{(\beta - \bar\beta) \cdot\nabla u}_{K} + \norm{ \Pi^{p-1} \big(  \bar\beta - \beta \big) \cdot\nabla u}_{K} + \norm{\Pi^{p-1}\gamma u}_{K}\\
    &\lesssim \norm{(\beta - \bar\beta) \cdot\nabla u}_{K}+ \norm{\Pi^{p-1}\gamma u}_{K}\\
        &\leq \norm{\beta-\bar\beta}_{L^\infty(K)}\norm{\nabla u}_{K} + \norm{\gamma}_{L^\infty(K)}\norm{u}_{K}\\
                                        &\leq h_K\abs{\beta}_{\mathrm{Lip}(K)}\norm{\nabla u}_{K} + \norm{\gamma}_{L^\infty(K)}\norm{u}_{K}\\
                                    &\leq (\abs{\beta}_{\mathrm{Lip}(K)} + \norm{\gamma}_{L^\infty(K)})\norm{u}_{K},
\end{align*}
where we have used the assumed regularity of $\beta$ and the inverse inequality.
\end{proof}

\begin{corollary}\label{th:arwellposed}
    The bilinear form $a_h(\cdot,\cdot)$ given in \eqref{eq:arah} is coercive on $\ITh$ and continuous on $\Vsh\times\ITh$, 
and hence \cref{ass:assumption_Th,ass:cont_Th} hold.
\end{corollary}
\begin{proof}
    Using integration by parts on the advection term we have for all $v_h\in\ITh$
\begin{align*} \label{eq:advection2}
    a_h(v_h,v_h) 
    =\ & \inner{(\gamma - \frac12\Div \beta)v_h,v_h}_{\Th} + \sum_{K\in\Th}\frac12\inner{(\beta\cdot\nf) v_h,v_h}_{\partial K}
      - \inner{(\beta\cdot\nf)\jmp{v_h},\avg{v_h}}_{\Fhi} \\
       &- \inner{(\beta\cdot\nf)v_h,v_h}_{\partial\Omega^-} + \frac12\sum_{F\in\Fhi}\int_F|\beta\cdot n_F|\jmp{v}^2.
\end{align*}
Now using that $\frac12 \jmp{v_h^2}=\avg{v_h}\jmp{v_h}$ the second and third term cancel out on inner facets and by summing the terms on the boundary facets we get
\begin{align*} 
    a_h(v_h,v_h) 
    =\ & \inner{(\gamma - \frac12\Div \beta)v_h,v_h}_{\Th} 
    + \frac12\int_{\partial\Omega}|\beta\cdot\nf|v_h^2 + \frac12\sum_{F\in\calF_h}\int_F|\beta\cdot n_F|\jmp{v}^2.
\end{align*}
With assumption \eqref{assumptiondiv} and \cref{lem:arnormeq} this proves coercivity.
Continuity follows from application of the Cauchy-Schwarz inequality and the definition of the norm $\Vert \cdot \Vert_{\Vsh}$.
\end{proof}

\subsubsection{Putting it all together}\label{sec:putar}
\begin{corollary}\label{cor:arput}
    Let $u\in H^{s+1}(\Omega)$ be the solution to the weak form of problem \eqref{eq:arpde} and $u_h$ the solution of the Trefftz DG problem \eqref{eq:Tarvar}.
Set $m = \min\{s,p\}$.
We have the following error estimate
\begin{equation}
    \Vert u - u_h \Vert_{\Vh} \lesssim \inf_{v_h \in \Vh} \Vert u - v_h \Vert_{\Vsh} \lesssim h^{m+1/2} |u|_{H^{m+1}(\Th)}.
\end{equation}
\end{corollary}
\begin{proof}
    Due to \cref{th:arwellposed} coercivity and continuity hold for the problem on the Trefftz space $\ITh\subset \Vh$, and as a result \cref{ass:assumption_Th,ass:cont_Th} hold.
    As discussed in \cref{sec:embtrefftz} we have that \cref{ass:inexacttrefftz} is satisfied with $\rho = 0$. 
    Hence, our choice of local operator \eqref{eq:arak} satisfies \cref{ass:locality}. 
    We have shown \cref{eq:quasiortho:c} in \cref{sec:arproto}.
    We can now apply \cref{cor:locality}, showing that \cref{ass:local} is satisfied.

    We now prove \cref{ass:local_strong}: For any $u\in \Vsh$ we have that 
    \begin{align*}
        \norm{\ATh u }_{\QThp} = 
        \sup_{q\in \IP^{p-1}(\Th)} \frac{\inner{h_K^{1/2} (\beta \cdot \nabla v_h + \gamma v_h),q}_\Th}{\norm{q}_{L^2(\Th)}}  
\leq \norm{h_K^{1/2} \beta \cdot \nabla v_h + \gamma v_h}_{L^2(\Th)}        \lesssim \TnormDG{u}, 
    \end{align*}
    using Cauchy–Schwarz inequality, triangle inequality, and $h_K\lesssim 1$.

    Further we have global and local consistency, i.e. $a_h(u_h,v_h) = a_h(u,v_h),\ \forall v_h\in \ITh$ and $A_K u_h = A_K u$ (in $\QKh^\prime$). 
    Hence, we can apply \cref{cor:cea} to obtain the error estimate.
\end{proof}

\subsubsection{Numerical example}
We choose $\Omega=(0,1)^2$ with a sequence of uniform triangular meshes and 
\begin{align}\label{eq:exar}
\beta=(-x_1,x_2)^T,\ 
\gamma=x_1+x_2,\
u_{\mathrm{ex}} = \sin\big(\pi(x_1\!+\!x_2)\big).
\end{align}
The right-hand side $f$ is constructed in order to manufacture the solution $u_{\mathrm{ex}}$.
Dirichlet boundary conditions match the exact solution.

In \cref{fig:exar} we show the convergence rate of the Trefftz DG method for the problem \eqref{eq:arpde} and compare it to the standard DG method for different polynomial degrees $p=3,4,5$.
For both methods we observe the expected convergence rates for the $\Vh$-error of $h^{p+1/2}$.
The $L^2$-error converges with the rate $h^{p+1}$, 

\begin{figure}[ht!]\centering
\resizebox{.8\linewidth}{!}{
\begin{tikzpicture}
	\begin{groupplot}[
	group style={
	group name={my plots},
	group size=2 by 1,
	horizontal sep=2cm,
	},
	legend style={
	legend columns=8,
	at={(0.8,-0.2)},
	draw=none
	},
	ymajorgrids=true,
	grid style=dashed,
	cycle list name=paulcolors2,
	]      

	\nextgroupplot[ymode=log,xmode=log,x dir=reverse, ylabel={$L^2$-error},xlabel={$h$}]
    \foreach \k in {3,4,5}{
    \addplot+[discard if not={p}{\k},discard if not={method}{et}] table [x=h, y=l2error, col sep=comma] {ex/ar2d.csv};
    \addplot+[discard if not={p}{\k},discard if not={method}{dg}] table [x=h, y=l2error, col sep=comma] {ex/ar2d.csv};
    }
    \addlegendimage{solid}
	\addplot[domain=0.06:1.0] {exp(-4*ln(1/x)-3.2)};
	\addplot[domain=0.06:1.0] {exp(-5*ln(1/x)-4.6)};
	\addplot[domain=0.06:1.0] {exp(-6*ln(1/x)-6.3)};

	\nextgroupplot[ymode=log,xmode=log,x dir=reverse, ylabel={$\Vh$-error},xlabel={$h$}]
	\foreach \k in {3,4,5}{
	\addplot+[discard if not={p}{\k},discard if not={method}{et}] table [x=h, y=dgerror, col sep=comma] {ex/ar2d.csv};
    \addplot+[discard if not={p}{\k},discard if not={method}{dg}] table [x=h, y=dgerror, col sep=comma] {ex/ar2d.csv};
	}
	\addlegendimage{solid}
	\addplot[dashed,domain=0.06:1.0] {exp(-3.5*ln(1/x)-0.5)};
	\addplot[dashed,domain=0.06:1.0] {exp(-4.5*ln(1/x)-2.0)};
	\addplot[dashed,domain=0.06:1.0] {exp(-5.5*ln(1/x)-4.5)};
    \legend{$\IT^3$,$\IP^3$,$\IT^4$,$\IP^4$,$\IT^5$,$\IP^5$,$\mathcal O(h^{p+1})$,{$\mathcal O(h^{p+1/2})$ for $p=3,4,5$},}
	\end{groupplot}
\end{tikzpicture}}
\vspace{-0.5em}
\caption{
    Convergence of the Trefftz DG method for the problem \eqref{eq:arpde} with exact solution \eqref{eq:exar}.
    The left plot shows the $L^2$-error and the right plot the $\Vh$-error.
    We compare the Trefftz DG method with the standard DG method, plotted with dashed lines.
    The black lines indicate the expected convergence rates.
}
\label{fig:exar}
\end{figure}

\subsection{Embedded Trefftz DG for the diffusion-advection-reaction equation}\label{sec:dar}
Let $f\in L^2(\Omega)$, $g_D\in H^{\frac12}(\partial\Omega)$ and set $V=H^1(\Omega)\cap H^2(\Th)$.
We consider the following boundary value problem for the diffusion--advection--reaction equation:
\begin{eqs}\label{eq:darpde}
-\mathrm{div}(\alpha \nabla u) + (\beta\cdot\nabla) u +\gamma u &= f  && \quad \text{in  }\Omega,\\
u&=g_{\mathrm D} && \quad\text{on } \partial\Omega. 
\end{eqs}
We consider the following DG discretization of problem \eqref{eq:darpde}:  
\begin{equation}\label{eq:darvar}
\text{Find } u_h \in \Vh \text{ such that }
a_h(u_h,v_h)=\ell_h(v_h) \quad \forall v_h \in \Vh,
\end{equation}
with a standard DG bilinear form $a_h: V_{*h} \times \Vh \to \mathbb{R}$ that combines the interior penalty method for the treatment of the diffusion term and the upwind DG method for the advection term,
\begin{eqs}\label{eq:darah}
a_h(w,v_h):=&
\inner{\alpha \nabla w, \nabla v_h}_{\Th} + \inner{\beta \cdot \nabla w, v_h}_{\Th} + \inner{\gamma w, v_h}_{\Th}\\
    &- \inner{\avg{\alpha   \nabla w \cdot\nf},\jmp{v_h}}_{\Fhi} 
    - \inner{\jmp{w},\avg{\alpha   \nabla v_h\cdot\nf}}_{\Fhi} 
    + \inner{\sigma\frac{\alpha_F}{h_F}\jmp{w},\jmp{v_h}}_{\Fhi}\\
    &- \inner{\alpha   \nabla w \cdot \nf, v_h}_{\partial\Omega} 
    - \inner{w, \alpha   \nabla v_h \cdot \nf}_{\partial\Omega} 
    + \inner{\sigma\frac{\alpha_F}{h_F} w, v_h}_{\partial\Omega}\\
    &- \inner{(\beta\cdot\nf)\jmp{w},\avg{v_h}}_{\Fhi} 
    + \frac12 \inner{|\beta\cdot \nf| \jmp{w},\jmp{v_h}}_{\Fhi}
    -\inner{(\beta\cdot\nf) w,v_h}_{\partial\Omega^-},
\end{eqs}
and the linear form $\ell_h:\Vh\to\IR$, 
defined by 
\begin{equation}\label{eq:darlh}
\ell_h(v_h):=
\inner{f,v_h}_{\Th} - \inner{g_{\mathrm D},\alpha   \nabla v_h \cdot \nf}_{\partial\Omega} + \inner{g_{\mathrm D},\sigma\frac{\alpha_F}{h_F} v_h}_{\partial\Omega} - \inner{g_{\mathrm D}\beta\cdot\nf,v_h}_{\partial\Omega^-}.
\end{equation}
The bilinear form $a_h(\cdot,\cdot)$ depends on the parameters $\sigma,\alpha_F>0$ that penalize the jumps of the function values.
The quantity $\sigma>0$ is a dimensionless constant independent of the diffusion coefficient $\alpha$, while $\alpha_F$ is a possibly weighted average of the diffusion parameter defined on each facet such that $\alphamin \leq \alpha_F\leq \norm{\alpha}_{L^{\infty}(K_1\cup K_2)}$ for all $F\in \calF_h^{\mathrm I}$ with $F=\partial K_1 \cap \partial K_2$, and $\alphamin \leq \alpha_F\leq \norm{\alpha}_{L^{\infty}(K)}$ for all $F\in \calF_h^{\mathrm B}$ with $F\subset\partial K \cap \partial\Omega$.
For the diffusion coefficient we assume that $\alpha \in W^{1,\infty}(\Th)$ and that it is strictly positive. 
We further assume $\beta\in \left[W^{1,\infty}(\Omega)\right]^{d},$ and $ \gamma \in L^{\infty}(\Omega)$.
We will not track the dependence on $\norm{\beta}_{L^\infty(\Th)}$ and $\norm{\gamma}_{L^\infty(\Th)}$.

To introduce the embedded Trefftz DG for the problem \eqref{eq:darpde} we follow the steps outlined in \cref{sec:embtrefftz}.
To this end we define the local operators
\begin{eqs}\label{eq:darak}
    &\AK  v_h = h_K(-\Div(\alpha \nabla v_h) + \beta \cdot \nabla v_h + \gamma v_h)
    \text{ and }
    \ell_K = h_K f.
\end{eqs}
As local test space we choose $\QKh = \IP^{p-2}(K)$. 
As for the advection-reaction problem, this choice is influenced be the highest order operator of the PDE, and is made with a prototype operator in mind, thus with the chosen $\QKh$ and the scaling in $\AK $ \cref{ass:local_strong} holds. 
Now, as in \cref{sec:embtrefftz}, the Trefftz space is given by \eqref{eq:IThproto}, resulting in
\begin{align*}
    \ITh&= \{v_h \in \Vh \mid \inner{\AK  v_h, q_h}_K=0,\ \forall q_h\in \IP^{p-2}(K),\ K\in\Th \}.
\end{align*}
The embedded Trefftz DG method for the diffusion-advection-reaction problem then reads:
\begin{eqs}\label{eq:Tdarvar}
    \text{Find } u_h \in \Vh 
    &\text{ such that } &&\inner{\AK  u_h, q_h}_K=\inner{h_K f,q_h}_K \quad \forall q_h\in \IP^{p-2}(K), K\in\Th
 \\ &\text{ and } &&a_h(u_h,v_h)=\ell_h(v_h) \quad \forall v_h \in \ITh,
\end{eqs}
with $a_h(\cdot,\cdot)$ given in \eqref{eq:darah} and $\ell_h(\cdot)$ given in \eqref{eq:darlh}.

For all $v\in \Vsh$ we define the following mesh-dependent norms:
\begin{eqs}\label{eq:darnorms}
\normDG{v}^2:=&	
\inner{\alpha\nabla v, \nabla v}_\Th + 
\gamma_0\norm{v}^2_{L^{2}(\Omega)}+ 
\sum_{F\in\Fh}\sigma\frac{\alpha_F}{h_F}\int_F \jmp{v}^2 +
\frac12\sum_{F\in\calF_h}\int_F\abs{\beta\cdot n_F}\jmp{v}^2,\\
\TnormDG{v}^2:=&\normDG{v}^2+
\sum_{K\in \Th}h_K\norm{\alpha^{\frac12}\nabla v\cdot \nf}^2_{L^2(\partial K)}+
\sum_{K\in \Th}\norm{\beta}_{L^{\infty}(K)}\norm{v}^2_{L^2(\partial K)},
\end{eqs}
and as before from the definition of the norms \cref{ass:locality} is obvious. 

\subsubsection{Coercivity on the Trefftz space}
In contrast to the previous example, the DG bilinear form is coercive.
Different proofs of the well-posedness can be found in the literature, 
many rely on the ``boundedness on orthogonal subscales'', see \cite[Lemma~2.30]{DiPietroErn}, requiring that (piecewise) partial derivatives of elements of $\Vh$ belong to $\Vh$.
This can generally not be expected from the Trefftz-like spaces, we thus refer to \cite[Theorem~4.3]{2408.00392} for a proof that avoids the assumptions on $\Vh$. 
The price is paid in an unfavorable dependence of the continuity bound on the P\'eclet number, for more details we refer to \cite[Section~4.4]{2408.00392}.
We summarize the well-posedness in the following theorem.

\begin{theorem}\label{th:darwellposed}
The bilinear form $a_h(\cdot,\cdot)$ given in \eqref{eq:darah}
is coercive on $\Vh$
and continuous on $V_{*h}\times \Vh$.
Hence, DG variational problem \eqref{eq:darvar} admits a unique solution $u_h\in \Vh$, for any subspace 
$\Vh\subseteq \IP^p(\Th)$.
The weak solution $u$ of the BVP \eqref{eq:darpde}
solves the variational problem \eqref{eq:darvar}, i.e.\ \eqref{eq:darvar} is consistent.
\end{theorem}

\subsubsection{The prototype operator for \texorpdfstring{\cref{lem:abstractneumann}}{}}\label{sec:darproto}
Similar to the advection-reaction case, we choose the prototype operator according to the highest order operator in $A_K$, setting
\begin{equation}\label{eq:darak0}
    A_{K,0} v_h = -h_K\Delta v_h.
\end{equation}
The kernel of the prototype operator $A_{K,0}$ are the harmonic polynomials.
From the theory on harmonic polynomials, cf. \cite[Prop. 5.5]{harmonicft}, we know that a suitable complementary space to the harmonic polynomials is given by 
\begin{equation}\label{eq:UKhdar}
    \UKh = |x|^2\IP^{p-2}(K).
\end{equation}

\begin{lemma}\label{lem:nonharmonicNE}
    For all $u\in \UKh$ we have
\begin{equation}\label{eq:nonharmominNE}
     \norm{\Lap u}^2_{K} \simeq  h_K^{-2}\norm{\nabla u}^2_{K} + h_K^{-4}\norm{u}^2_{K},   
\end{equation}
    where the constants depend only on the shape regularity of the mesh and the polynomial degree.
    And therefore
    \begin{equation*}
        \norm{A_{K,0}u}_{\QThp}^2 \gtrsim \sum_{K\in\Th} \norm{\nabla u}^2_{K} + h_K^{-2}\norm{u}^2_{K}     
        \gtrsim \normDG{u}^2\qquad \forall u\in \Uh.
    \end{equation*}
\end{lemma}
\begin{proof}
    Let $B_1$ be the unit ball and let $B_c$ be a ball of radius $c\simeq 1$. We have for any $u \in \IP^{k}(B_1\cup B_c)$  
    \begin{equation*}
         \norm{\Lap u}^2_{B_1} \simeq  \norm{\nabla u}^2_{B_1} + \norm{u}^2_{B_1} \simeq \norm{\nabla u}^2_{B_c} + \norm{u}^2_{B_c} \simeq \norm{\Lap u}^2_{B_c}.     
    \end{equation*}
    Let $b_k$ and $B_K$ be balls such that $b_K\subset K \subset B_K$ and $c = \diam(b_K) / \diam(B_K) \simeq 1$. Without loss of generality we may assume that $b_K$ and $B_K$ are centered at the same point. Using a linear mapping $\Phi_K: B_1 \to B_K$ we get by standard scaling arguments
    \begin{equation*}
        \norm{\Lap u}_{K} \lesssim \frac{1}{h_K^2} \norm{\Lap (u\circ\Phi_K) }_{B_1} 
                               \simeq  \frac{1}{h_K^2}(\norm{\nabla (u\circ\Phi_K) }_{B_c} + \norm{u\circ\Phi_K }_{B_c}) \lesssim \frac{1}{h_K}\!\norm{\nabla u}_{K} + \frac{1}{h_K^2}\!\norm{u}_{K}.
                           \end{equation*}
    Analogously we get
    \begin{equation*}
        \norm{\Lap u}_{K} \gtrsim \frac{1}{h_K^2} \norm{\Lap (u\circ\Phi_K) }_{B_c} 
                               \simeq  \frac{1}{h_K^2}(\norm{\nabla (u\circ\Phi_K) }_{B_1} + \norm{u\circ\Phi_K }_{B_1}) 
                               \gtrsim \frac{1}{h_K}\!\norm{\nabla u}_{K} + \frac{1}{h_K^2}\!\norm{u}_{K}.
   \end{equation*}

    The final estimate follows from the definition of the norm \eqref{eq:darnorms} and triangle- and trace inequality.
\end{proof}

Using \cref{lem:nonharmonicNE} we get the following corollary, allowing us to apply \cref{lem:abstractneumann}.
\begin{corollary}\label{cor:darneumann}
    The operators $\AK $ and $A_{K,0}$ for the advection-reaction-diffusion problem, defined in \eqref{eq:darak} and \eqref{eq:darak0}
    satisfy the assumptions of \cref{lem:abstractneumann}.
\end{corollary}
\begin{proof}

    Let $\omega = \frac{1}{\Pi^0 \alpha}$.
    For $\alpha$ element-wise smooth, using the product rule and standard inequalities, 
    we have for all $u \in H^2(K)$
    \[ 
    \begin{aligned}
        &\frac{1}{\omega h_K} \|\omega \AK   u - A_{K,0}u \|_{K}   \le
        \norm{\Div(( \alpha\!-\! \Pi^0\alpha ) \nabla u)}_{K} + \norm{\beta}_{L^\infty(K)} \norm{u}_{H^1(K)} + \norm{\gamma}_{L^\infty(K)} \norm{u}_{K} \\
        &\quad\le \norm{\alpha\!-\! \Pi^0\alpha }_{L^\infty(K)}\norm{\Delta u}_{K} 
+ C\norm{u}_{H^1(K)}
        \lesssim h_K\norm{\nabla \alpha }_{L^\infty(K)} \norm{\Delta u}_{K} + C\norm{u}_{H^1(K)},
    \end{aligned}
    \]
    with $C=\| \nabla \alpha \|_{L^\infty(K)}\!+\!\norm{\beta}_{L^\infty(K)}\! +\! \norm{\gamma}_{L^\infty(K)}$.
    Now using \cref{lem:nonharmonicNE} we obtain that for any $u\in \UKh$ 
    \begin{align*}
        \norm{\AK  u - A_{K,0} u}_{K} &\lesssim C (h_K^2 \norm{\Delta u}_{K} + h_K \norm{u}_{H^1(K)})
                                       \\ &\lesssim C h_K^2 \norm{\Delta u}_{K} = C h_K \norm{A_{K,0} u}_{\QKh'},
    \end{align*}
    where we have used for the last equality that $\Lap u\in \QKh$.
    From this we obtain
    \begin{equation*}
    \frac{\norm{\AK  u - A_{K,0} u}_{\QKh'}}{\norm{A_{K,0}u}_{\QKh'}} \lesssim C h_K,
    \end{equation*}
    which for $h_K$ small enough implies \cref{eq:nearlycontant} as $C$ does not depend on $h_K$.
\end{proof}

\subsubsection{Putting it all together}\label{sec:darput}

\begin{corollary}\label{cor:darput}
    Let $u\in H^{s+1}(\Omega)$ be the solution to the weak form of problem \eqref{eq:darpde} and $u_h$ the solution of the Trefftz DG problem \eqref{eq:Tdarvar}.
Set $m = \min\{s,p\}$.
Under the assumptions of \cref{th:darwellposed} we have the following error estimate
\begin{equation}
    \Vert u - u_h \Vert_{\Vh} \lesssim \inf_{v_h \in \Vh} \Vert u - v_h \Vert_{\Vsh} \lesssim h^{m} |u|_{H^{m+1}(\Th)}. 
\end{equation}
\end{corollary}
\begin{proof}
    Due to \cref{th:darwellposed} coercivity and continuity hold for the problem on the Trefftz space $\ITh\subset \Vh$, and as a result \cref{ass:assumption_Th,ass:cont_Th} hold.
    As discussed in \cref{sec:embtrefftz} we have that \cref{ass:inexacttrefftz} is satisfied with $\rho = 0$. 
    We have shown \cref{eq:quasiortho:c} in \cref{sec:darproto}.
    Now applying \cref{cor:locality}, shows that \cref{ass:local} is satisfied.

    We now prove \cref{ass:local_strong}: For any $u\in \Vsh$ we have that 
    \begin{align*}
        \norm{\ATh u }_{\QThp}^2 &= 
        \sum_{K\in \Th} \sup_{q\in \IP^{p-2}(K)} \frac{\inner{h_K(-\Div(\alpha \nabla u) + \beta \cdot \nabla u + \gamma u) ,q}_K^2}{\norm{q}_{L^2(K)}^2}  \\
     &\lesssim \sum_{K\in\Th} \sup_{q\in \IP^{p-2}(K)}  
\frac{ (h_K^2\norm{\nabla q}_K^2 + h_K\norm{q}_{\partial K}^2 +\norm{q}_{L^2(K)}^2) }{\norm{q}_{L^2(K)}}\\
     &\qquad \qquad (\norm{\alpha^{1/2}\nabla u}_K^2 + h_K\norm{\nabla u\cdot n}_{\partial K}^2 + \gamma_0\norm{u}_{L^2(K)}^2  )  
                                   \\
       &\leq \TnormDG{u}^2, 
    \end{align*}
    using Cauchy–Schwarz inequality, triangle inequality, and $h_K\lesssim 1$.

    Further we have global and local consistency.
    Hence, we can apply \cref{cor:cea} to obtain the error estimate.
\end{proof}

\begin{corollary}

Let the assumptions of \cref{cor:darput} hold. Additionally, assume that 
the boundary of $\Omega$ is sufficiently smooth or that $\Omega$ is convex, such that $L^2-H^2$-regularity holds. There holds the error estimate
\begin{align*}
    \Vert u - u_h \Vert_{\Omega} \lesssim h^{m+1} |u|_{H^{m+1}(\Th)}. 
\end{align*}
\end{corollary}
\begin{proof}
We make use of \cref{th:aubinnitsche} with the choice $H = L^2$, $W = V$ and $\Wh = \Vh$ with $\norm{\cdot}_{\Wsh}= \norm{\cdot}_{\Vsh}$. With this choice we have that the continuity condition \eqref{eq:a-bit-cont} holds.
Since we considered a symmetric interior penalty formulation for the diffusion operator, adjoint consistency \eqref{eq:a-bit-consistent} is fulfilled.
Further, as can be shown by standard inverse estimates, the norms $\Vert \cdot \Vert_{\Vh}$ and $\Vert \cdot \Vert_{\Vsh}$ are equivalent on $\Vh$, and $\Vert \cdot \Vert_{\Vsh}$ is stronger than $\Vert \cdot \Vert_\V$ on $V$.
The bound \eqref{eq:H-reg-bound2} holds true due to standard approximation results and the assumed $L^2-H^2$-regularity. Finally, for any $z\in V \cap H^2(\Omega)$ we have $\AThnorm{z}{} \lesssim h \norm{z}_{H^2(\Omega)} \lesssim h \norm{a(\cdot,z)}_{L^2(\Omega)}$. This implies that \eqref{eq:H-reg-bound} holds, cf. \cref{lemma:H-reg-bound}.
\end{proof}

\subsubsection{Numerical example}\label{sec:darex}
We consider the problem \eqref{eq:darpde} on the unit square $\Omega = (0,1)^2$ with a sequence of uniform triangular meshes.
The PDE coefficients and the solution are chosen as
\begin{align}\label{eq:darex}
\alpha=1\!+\!x_1\!+\!x_2,\
\beta=\begin{pmatrix} \sin x_1\\\sin x_2\end{pmatrix},\
\gamma=\frac{4}{1\!+\!x_1\!+\!x_2\!+\!x_3},\ 
u_{\mathrm{ex}} = \sin(\pi(x_1\!+\!x_2\!)).
\end{align}
The right-hand side $f$ is constructed in order to manufacture the solution $u_{\mathrm{ex}}$ in \eqref{eq:darex}.
Dirichlet boundary conditions are imposed on the entire boundary of the domain.
The penalization parameters is chosen as $\sigma =50p^2$.

In \cref{fig:darex} we show the convergence of the Trefftz DG method for the problem \eqref{eq:darpde} with exact solution \eqref{eq:darex}.
We compare the Trefftz DG method with the standard DG method for different polynomial degrees $p=3,4,5$.
We observe the expected convergence rates for the $L^2$-error and the $\Vh$-error.

\begin{figure}[ht!]\centering
\resizebox{.8\linewidth}{!}{
\begin{tikzpicture}
	\begin{groupplot}[
	group style={
	group name={my plots},
	group size=2 by 1,
	horizontal sep=2cm,
	},
	legend style={
	legend columns=8,
	at={(0.8,-0.2)},
	draw=none
	},
	ymajorgrids=true,
	grid style=dashed,
	cycle list name=paulcolors2,
	]      
	\nextgroupplot[ymode=log,xmode=log,x dir=reverse, ylabel={$L^2$-error},xlabel={$h$}]
    \foreach \k in {3,4,5}{
    \addplot+[discard if not={p}{\k},discard if not={method}{et}] table [x=h, y=l2error, col sep=comma] {ex/dar2d.csv};
    \addplot+[discard if not={p}{\k},discard if not={method}{dg}] table [x=h, y=l2error, col sep=comma] {ex/dar2d.csv};
    }
    \addlegendimage{solid}
    \addplot[domain=0.06:1.0] {exp(-4*ln(1/x)-3.5)};
    \addplot[domain=0.06:1.0] {exp(-5*ln(1/x)-4.3)};
    \addplot[domain=0.06:1.0] {exp(-6*ln(1/x)-6.5)};

	\nextgroupplot[ymode=log,xmode=log,x dir=reverse, ylabel={$\Vh$-error},xlabel={$h$}]
	\foreach \k in {3,4,5}{
	\addplot+[discard if not={p}{\k},discard if not={method}{et}] table [x=h, y=dgerror, col sep=comma] {ex/dar2d.csv};
    \addplot+[discard if not={p}{\k},discard if not={method}{dg}] table [x=h, y=dgerror, col sep=comma] {ex/dar2d.csv};
	}
	\addlegendimage{solid}
	\addplot[dashed,domain=0.06:1.0] {exp(-3*ln(1/x)-0.2)};
	\addplot[dashed,domain=0.06:1.0] {exp(-4*ln(1/x)-1.5)};
	\addplot[dashed,domain=0.06:1.0] {exp(-5*ln(1/x)-2.9)};
	\legend{$\IT^3$,$\IP^3$,$\IT^4$,$\IP^4$,$\IT^5$,$\IP^5$,$\mathcal O(h^{p+1})$,{$\mathcal O(h^p)$ for $p=3,4,5$},}
	\end{groupplot}
\end{tikzpicture}}
\vspace{-0.5em}
\caption{
    Convergence of the Trefftz DG method for the problem \eqref{eq:darpde} with exact solution \eqref{eq:darex}.
    The left plot shows the $L^2$-error and the right plot the $\Vh$-error.
    We compare the Trefftz DG method with the standard DG method, plotted with dashed lines.
    The black lines indicate the expected convergence rates.
}
\label{fig:darex}
\end{figure}

\subsection{Quasi-Trefftz DG for the diffusion equation}\label{sec:qtdg}
In this section we consider the diffusion problem with the quasi-Trefftz DG method, briefly discussed in \cref{sec:qtrefftz}.
The quasi-Trefftz DG method for general elliptic problems was introduced and analyzed in \cite{2408.00392}.
We show that the quasi-Trefftz DG method can also be analyzed using the presented framework.
Furthermore, using the framework we extend the analysis by presenting the first optimal error bounds in the $L^2$-norm.

We consider the following boundary value problem 
\begin{eqs}\label{eq:dpde}
-\mathrm{div}(\alpha \nabla u) &= f  && \quad \text{in  }\Omega,\\
u&=g_{\mathrm D} && \quad\text{on } \partial\Omega,
\end{eqs}
with $\alpha\in W^{1,\infty}(\Th)$, $f\in L^2(\Omega)$ and $g_{\mathrm D}\in H^{1/2}(\partial\Omega)$.
The standard symmetric interior penalty DG discretization of problem \eqref{eq:dpde} is easily obtained by setting $\beta = \gamma = 0$ in the bilinear form $a_h(\cdot,\cdot)$ defined in \eqref{eq:darah} and the linear form $l_h(\cdot)$ defined in \eqref{eq:darlh}.
The analysis is carried out using the $\Vh$-norm defined in \eqref{eq:darnorms} with $\beta = \gamma_0 = 0$.
We note that the symmetric interior penalty DG discretization is well-posed.

The quasi-Trefftz space $\ITh$ is defined following \eqref{eq:qtspace}, giving
\begin{equation}
    \ITh=\big\{ v\in \Vh \mid D^{\bi} \Div(\alpha\nabla v) (x_K) = 0  \quad \forall \bi\in \IN^d_0,\ |\bi|\leq p-2,\ K\in\Th\big\}.
\end{equation}

Finally, the quasi-Trefftz DG method then reads as
\begin{eqs}\label{eq:qtdvar}
    \text{Find } u_h \in \Vh 
     &\text{ such that } && D^{\bi} \Div(\alpha \nabla u_h)(x_K)  = D^\bi f(x_K) \quad \forall \bi\in \IN^d_0, |\bi|\leq p\!-\!2,
K\in \Th,  \\ &\text{ and } &&a_h(u_h,v_h)=\ell_h(v_h) \quad \forall v_h \in \ITh.
\end{eqs}
We require the PDE coefficient $\alpha$ and right-hand side $f$ to be element-wise sufficiently smooth, such that the quasi-Trefftz method is well-defined.
To apply the framework from \cref{sec:framework} to quasi-Trefftz DG methods, we define 
\begin{equation}\label{eq:dakqt}
        A_K v = (h^{\frac{3}{2} + |\bi|}D^{\bi} \Div(\alpha \nabla v)(x_K) )_{\bi},
        \quad\text{ and } \quad
        \ell_K = (h^{\frac{3}{2} + |\bi|} D^{\bi}f_K(x_K))_\bi,
\end{equation}
with $\QKh = \IR^{|\{\bi:|\bi|\leq p-2\}|}$.
We recall that multi-indices are denoted by $\bi:=(i_{1},\ldots,i_{d}) \in \IN_0^{d}$, their length $|\bi|:=i_{1}+\cdots+i_{d}$ and here $(\cdot)_{\bi}$ denotes the vector over the multi-indices $\bi$ with $|\bi|\leq p-2$, with any (but fixed) ordering or entries.
We keep the $\normDG{\cdot}$-norm as defined in \eqref{eq:darnorms} and define the norm 
\begin{align*}
    \norm{v}_{\Vsh}^2 := \normDG{v}^2 + \sum_{K\in\Th} \sum_{|\bi|\leq p-2} h_K^{2|\bi|-1} \norm{D^\bi v}^2_{C^0(K)}, 
\end{align*}
where $\norm{v}_{C^0(K)}:=\sup_{x\in K}|v(x)|.$

\subsubsection{The prototype operator for \texorpdfstring{\cref{lem:abstractneumann}}{}}\label{sec:qtproto}
We follow along \cref{sec:darproto}, choosing in virtue a similar prototype operator, adapted to the $Q_h$ space of the quasi-Trefftz DG method, resulting in
\begin{equation}\label{eq:dak0qt}
    A_{K,0}: \UKh \to \IR^{|\{\bi:|\bi|\leq p-2\}|} \text{ with } A_{K,0} v_h = (-h^{\frac{3}{2} + |\bi|}D^{\bi} \Delta v_h(x_K) )_{\bi}.
\end{equation}
As in \cref{sec:darproto}, the kernel of the prototype operator $A_{K,0}$ are still the harmonic polynomials and a complementary space is constructed as in \eqref{eq:UKhdar}.

\begin{lemma} \label{lem::lapa_m_lapavqt} 
    For any $K\in\Th$ and $\alpha\in C^{p-1}(K)$ we have for all $u\in \IP^{p}(K)$ that
    \begin{align*}
        \| \omega \AK  u - A_{K,0} u \|_{\IR^{|\{\bi:|\bi|\leq p-2\}|}} 
       &\lesssim 
       h_K\norm{\alpha}_{C^{p-1}(K)}\left( h_K \norm{\Delta u}_K
       + \norm{\nabla u}_K \right), 
    \end{align*}
     where $\omega = \frac{1}{\Pi^0 \alpha}$.
\end{lemma}
\begin{proof}
    For $\alpha$ element-wise smooth, using the product rule and standard inequalities, 
    we have for all $u \in \IP^{p}(K)$ and all $|\bi|\leq p-2$ that by using the Leibniz product rule
    \begin{align*}
       &\norm{h_K^{\frac{3}{2}+|\bi|}D^\bi \Div(( \alpha\!-\! \Pi^0\alpha ) \nabla u)}_{C^0(K)} \\
       &\lesssim  \sum_{\bell\leq \bi}\Big(\norm{h_K^{\frac{3}{2}+|\bi|}D^{\bell} ( \alpha\!-\! \Pi^0\alpha ) D^{\bi-\bell}\Delta u))}_{C^0(K)} 
+\norm{h_K^{\frac{3}{2}+|\bi|}D^\bell \nabla  \alpha  D^{\bi-\bell}\nabla u))}_{C^0(K)} \Big) \\
       &\lesssim \sum_{\bell\leq \bi}\Big(\norm{h_K^{|\bell|} D^{\bell} ( \alpha\!-\! \Pi^0\alpha ) h_K^{\frac{3}{2}+|\bi|-|\bell|}D^{\bi-\bell}\Delta u))}_{C^0(K)}\!\!+\!
\norm{h_K^{1+|\bell|}D^\bell \nabla  \alpha  h_K^{\frac{1}{2}+|\bi|-|\bell|} D^{\bi-\bell}\nabla u))}_{C^0(K)} \Big) \\
       &\lesssim 
    \norm{( \alpha\!-\! \Pi^0\alpha ) h_K^{\frac{3}{2}+|\bi|}D^{\bi}\Delta u))}_{C^0(K)} 
       +\sum_{\mathbf{0} < \bell\leq \bi}\norm{h_K^{|\bell|} D^{\bell} \alpha  h_K^{\frac{3}{2}+|\bi|-|\bell|}D^{\bi-\bell}\Delta u))}_{C^0(K)} \\
       &\qquad+ \sum_{\bell\leq\bi} \norm{h_K^{1+|\bell|}D^\bell \nabla  \alpha  h_K^{\frac{1}{2}+|\bi|-|\bell|} D^{\bi-\bell}\nabla u))}_{C^0(K)}.
       \end{align*}
       Using best approximation properties and collecting the terms in $\alpha$ we can continue to estimate
       \begin{align*}
        &\norm{h_K^{\frac{3}{2}+|\bi|}D^\bi \Div(( \alpha\!-\! \Pi^0\alpha ) \nabla u)}_{C^0(K)} \\
       &\lesssim \ \norm{\alpha\!-\!\Pi^0\alpha}_{L^\infty(K)} h_K^{\frac{3}{2}+|\bi|} \norm{D^\bi \Lap u}_{C^0(K)} 
       + h_K\norm{\alpha}_{C^{p-2}(K)} \sum_{\mathbf{0}<\bell\leq\bi} h_K^{\frac{3}{2}+|\bi|-|\bell|} \norm{D^{\bi-\bell}\Delta u}_{C^0(K)} \\
       &\qquad 
       + h_K\norm{\alpha}_{C^{p-1}(K)} \sum_{\bell\leq\bi} h_K^{\frac{1}{2}+|\bi|-|\bell|} \norm{D^{\bi-\bell}\nabla u}_{C^0(K)} \\
       &\lesssim h_K\norm{\alpha}_{C^{p-1}(K)}\left( \sum_{\bell\leq\bi} h_K^{\frac{3}{2}+|\bi|-|\bell|} \norm{D^{\bi-\bell}\Delta u}_{C^0(K)} 
       + \sum_{\bell\leq\bi} h_K^{\frac{1}{2}+|\bi|-|\bell|} \norm{D^{\bi-\bell}\nabla u}_{C^0(K)} \right)
       .
    \end{align*}
    By standard scaling arguments we have 
    \[    \sum_{\bell\leq\bi} h_K^{\frac{3}{2}+|\bi|-|\bell|} \norm{D^{\bi-\bell}\Delta u}_{C^0(K)} 
       + \sum_{\bell\leq\bi} h_K^{\frac{1}{2}+|\bi|-|\bell|} \norm{D^{\bi-\bell}\nabla u}_{C^0(K)} \lesssim \norm{\Delta u}_K
       + \norm{\nabla u}_K,
    \] 
    for $u\in \IP^p(K)$, hence the statement follows. 
\end{proof}

\begin{lemma}\label{lem:nonharmonicNEqt}
    For all $u\in \UKh$ we have
    \begin{equation*}
        \norm{A_{K,0}u}_{\QThp}^2 \gtrsim \sum_{K\in\Th} \norm{\nabla u}^2_{K} + h_K^{-2}\norm{u}^2_{K}     
        \gtrsim \normDG{u}^2.
    \end{equation*}
\end{lemma}
\begin{proof}
    By standard scaling arguments we have
    \begin{align*}
        \norm{A_{K,0}u}_{\QThp}^2 &  \simeq \sum_{K\in\Th}h_K^2\norm{\Lap u}_K^2.
    \end{align*}
    Using \cref{lem:nonharmonicNE} we obtain the result.
\end{proof}

\begin{corollary}\label{cor:darneumannqt}
    The operators $\AK $ and $A_{K,0}$ for the diffusion problem, defined in \eqref{eq:dakqt} and \eqref{eq:dak0qt}
    satisfy the assumptions of \cref{lem:abstractneumann}.
\end{corollary}
\begin{proof}
    Let $C=\| \alpha \|_{C^{p-1}(\Omega)}$.
    From \cref{lem::lapa_m_lapavqt} and \cref{lem:nonharmonicNE} we obtain that for any $u\in \UKh$ 
    \begin{align*}
        \norm{\AK  u - A_{K,0} u}_{K} &\lesssim C (h_K^2 \norm{\Delta u}_{K} + h_K \norm{u}_{H^1(K)})
                                   \\ &\lesssim C h_K^2 \norm{\Delta u}_{K}  \simeq 
                                       C h_K\norm{A_{K,0} u}_{\QKh'},
    \end{align*}
    where we have used for the last equality that $\Lap u\in\IP^{p-2}(K)$.
    For $h_K$ small enough this implies \cref{eq:nearlycontant} as $C$ does not depend on $h_K$.
\end{proof}

\subsubsection{Putting it all together}\label{sec:qtput}

\begin{corollary}\label{cor:qtput}
    Let $u\in C^{p+1}(\Th)\cap V$ be the solution to the weak form of problem \eqref{eq:dpde} and $u_h$ the solution of the Trefftz DG problem \eqref{eq:qtdvar}.
Under the assumptions of \cref{th:darwellposed} we have the following error estimate
\begin{equation}
    \Vert u - u_h \Vert_{\Vh} \lesssim \inf_{v_h \in \Vh} \Vert u - v_h \Vert_{V_{*h}} \lesssim h^{p} |u|_{C^{p+1}(\Th)} .
\end{equation}
\end{corollary}
\begin{proof}
    Due to \cref{th:darwellposed} coercivity and continuity hold for the problem and as a result \cref{ass:assumption_Th,ass:cont_Th} hold.
    We have that \cref{ass:inexacttrefftz} is satisfied with $\rho = 0$. 
    The norm $\TnormDG{\cdot}$ is chosen such that our local operator is continuous. 
    We have shown \cref{eq:quasiortho:c} in \cref{sec:qtproto}.
    Now applying \cref{cor:locality}, shows that \cref{ass:local,ass:local_strong} are satisfied.
    Further we have global and local consistency.
    Hence, we can apply \cref{cor:cea} to obtain the error estimate.
\end{proof}

\begin{corollary}
Let the assumptions of \cref{cor:qtput} hold. Additionally, assume that 
the boundary of $\Omega$ is sufficiently smooth or that $\Omega$ is convex, such that $L^2-H^2$-regularity holds. There holds the error estimate
\begin{align*}
    \Vert u - u_h \Vert_{\Omega} \lesssim h^{p+1} |u|_{C^{p+1}(\Th)}. 
\end{align*}
\end{corollary}
\begin{proof}

We make use of \cref{th:aubinnitsche} with the choice $H = L^2$, $W = V$ and $\Wh = \Vh$ with $$\norm{\cdot}_{\Wsh}^2= \norm{\cdot}_{\Vh}^2
+\sum_{K\in \Th}h_K\norm{\alpha^{\frac12}\nabla v\cdot \nf}^2_{L^2(\partial K)}.
$$ With this choice we have that the continuity condition \eqref{eq:a-bit-cont} holds.
Since we considered a symmetric interior penalty formulation for the diffusion operator, adjoint consistency \eqref{eq:a-bit-consistent} is fulfilled.
Further, as can be shown by standard inverse estimates, the norms $\Vert \cdot \Vert_{\Vh}$ and $\Vert \cdot \Vert_{\Vsh}$ are equivalent on $\Vh$, and $\Vert \cdot \Vert_{\Vsh}$ is stronger than $\Vert \cdot \Vert_\V$ on $\V$.
The bound \eqref{eq:H-reg-bound2} holds true due to standard approximation results and the assumed $L^2-H^2$-regularity.

It remains to show \eqref{eq:H-reg-bound}. Let $P:H^2(\Th)\to \Vh$ be the $H^2$-orthogonal projection onto $\Vh$. Let $\widetilde{ \ATh} = \ATh P$. Let $z_h = z_\IL + z_\IT \in \Uh \oplus \ITh$ be the solution of 
\[
    \left(
    \begin{array}{c@{\quad\quad}c}
         \langle \widetilde\ATh  \cdot, q_h \rangle &  0  \\
        a_h(\cdot,v_h) & a_h(\cdot,v_h) 
    \end{array}
    \right) 
    \left(
     \begin{array}{c}
        z_\IL  \\
        z_{\IT}
     \end{array}
     \right)
    = \left(
    \begin{array}{c}
        0 \\
        a(z,v_h)
    \end{array}
    \right)
    , \quad \forall q_h\in \QKh, \forall v_h\in  \ITh.
\]

The operator $\widetilde\ATh$ is coercive as $P$ is the identity for polynomials, hence \cref{ass:local} holds for $\widetilde\ATh$.
The continuity of $\widetilde\ATh$ is measured with respect to the norm 
\begin{align*}
    \norm{v}_{\widetilde{V_{*h}}}^2 = \norm{v}_{\Wsh}^2 
    + \norm{h_K v}_{H^2(\Th)}^2.
\end{align*}
For any polynomial $w\in\IP^p(\Th)$ we have by standard scaling arguments that 
$$\norm{\ATh w}_{\IR^{|\{\bi:|\bi|\leq p-2\}|}} \leq \norm{h_K w}_{H^2(\Th)}.$$
For any $v\in V\cap H^2(\Th)$ we have that $w=Pv\in \IP^p(\Th)$ satisfies 
$\norm{w}_{H^2(\Th)} \leq \norm{v}_{H^2(\Th)}$, hence
\begin{equation}\label{eq:qth2bound}
    \norm{\widetilde\ATh v }_{\IR^{|\{\bi:|\bi|\leq p-2\}|}} = \norm{\ATh Pv}_{\IR^{|\{\bi:|\bi|\leq p-2\}|}} \lesssim  \norm{h_KPv}_{H^2(\Th)} \leq \norm{h_K v}_{H^2(\Th)}, 
\end{equation}
which shows \cref{ass:local_strong}.

Hence, by the reasoning of \cref{cor:qtput} we can apply \cref{cor:cea} for $\widetilde\ATh$. 
Keeping in mind the consistency error committed by $\widetilde\ATh$ we obtain
\begin{equation*}
    \norm{z-z_h}_{\Vh} \!\lesssim\! \inf_{v_h \in \Vh}\!\norm{z - v_h}_{\widetilde{V_{*h}}} + \norm{\widetilde\ATh z}_{\QThp}.
\end{equation*}
By standard best approximation estimates and \eqref{eq:qth2bound} we can bound 
\begin{equation*}
    \inf_{v_h \in \Vh}\!\norm{z - v_h}_{\widetilde{V_{*h}}} + \norm{\widetilde\ATh z}_{\QThp} \lesssim h \norm{z}_{H^2(\Th)}.
\end{equation*}
which shows \eqref{eq:H-reg-bound}.
Thus, we can apply \cref{th:aubinnitsche}.
\end{proof}
These results for the quasi-Trefftz method have been numerically verified in \cite{2408.00392}.

\section*{Acknowledgements}
This research was funded in part by the Austrian Science Fund (FWF) 
\href{https://doi.org/10.55776/F65}{10.55776/F65} and 
\href{https://doi.org/10.55776/ESP4389824}{10.55776/ESP4389824}.
For open access purposes, the authors have applied a CC BY public copyright license to any author-accepted manuscript version arising from this submission.

\bibliographystyle{abbrvurl}
\bibliography{bib.bib}
\end{document}